\documentclass[12pt, reqno]{amsart}
\usepackage{amsmath}
\usepackage{amssymb}
\usepackage{amsfonts}
\usepackage{amsthm}
\usepackage{amscd}
\usepackage{relsize}
\usepackage{setspace}
\usepackage{geometry}
\usepackage{enumitem}
\usepackage{url}
\usepackage{xspace}
\usepackage{mathrsfs}
\usepackage{dsfont}
\usepackage{lmodern}
\usepackage{xcolor}
\usepackage[utf8]{inputenc}
\usepackage{mathtools}

\usepackage{hyperref}

\usepackage[textsize=footnotesize,textwidth=15ex]{todonotes}

\newtheorem{thm}{Theorem}[section]
\newtheorem{lem}[thm]{Lemma}
\newtheorem{prop}[thm]{Proposition}
\newtheorem{cor}[thm]{Corollary}

\theoremstyle{definition}
\newtheorem{remark}[thm]{Remark}
\newtheorem{example}[thm]{Example}

\numberwithin{equation}{section}

\newcommand{\FF}{\mathbf{F}}

\newcommand{\NN}{\mathbf{N}}

\newcommand{\ZZ}{\mathbf{Z}}

\newcommand{\balpha}{\boldsymbol{\alpha}}
\newcommand{\bbeta}{\boldsymbol{\beta}}

\DeclareMathOperator{\Aut}{Aut}
\DeclareMathOperator{\Alt}{Alt}
\DeclareMathOperator{\Sym}{Sym}

\begin{document}

\title[Fractal anti-tori]{Fractal anti-tori}

\author{Pierre-Emmanuel \textsc{Caprace}}
\author{Justin \textsc{Vast}}

\date{July 20, 2026}

\thanks{JV is a F.R.S.-FNRS Research Fellow; PEC and JV are supported in part by the FWO and the F.R.S.-FNRS under the EOS programme (project ID 40007542).}

\begin{abstract}
Let $\Gamma$ be a group acting properly and cocompactly on the product of two trees $T_1$ and $T_2$. 
An anti-torus  is a non-periodic flat plane in $T_1 \times T_2$ that is the convex hull of two secant periodic lines. That notion was introduced by Dani Wise as a tool to show that $\Gamma$ is irreducible. We establish a new criterion ensuring the existence of anti-tori, and use it to prove that if $\Gamma$ is an $S$-arithmetic lattice in a product of simple algebraic groups of rank one, then $T_1\times T_2$ contains anti-tori. As a byproduct, we obtain a sufficient condition ensuring that a group defined by a bi-reversible automaton contains non-abelian free sub-semigroups. We also introduce a new class of irreducible lattices acting regularly on the vertex set of a product of two trees, and containing anti-tori that are fractal aperiodic tilings  of the plane. This establishes a connection between lattices in products of trees and substitution tilings. 
\end{abstract}

\maketitle

%
%
%\begin{flushright}
%\begin{minipage}[t]{0.55\linewidth}\itshape\small
%\dots le monde naturel pénètre dans le spirituel, lui sert de pâture et concourt ainsi à opérer cet amalgame indéfinissable\dots
%
%\vspace{2mm}
%
%\hfill\upshape \textemdash C. Baudelaire,  \emph{\OE uvres complètes}, 1931.
%\end{minipage}
%\end{flushright}
%

\section{Introduction}

Let $X$ be a CAT(0) space and $\Gamma$ be a group acting properly and cocompactly  on $X$ by isometries. We say that a flat $P$ in $X$ is \textbf{periodic} if its stabilizer in $\Gamma$ acts cocompactly on $P$. 
An \textbf{anti-torus} is a non-periodic flat in $X$ which contains two secant periodic lines. 
In other words, an anti-torus is a $2$-dimensional flat $P$ which is the convex hull of two non-parallel geodesic lines, that are translation axes for two hyperbolic isometries $\gamma_1, \gamma_2 \in \Gamma$, and such that the stabilizer of $P$ in $\Gamma$ does not act cocompactly on $P$. We say that the anti-torus $P$ is \textbf{spanned} by $\gamma_1$ and $\gamma_2$. The  concept of an anti-torus was introduced by D.~Wise (see \cite{Wise_PhD} and \cite{Wise_CSC}), who identified it as an irreducibility criterion in case $X$ is a product of trees  (see \cite[Proposition~4.11]{Cap_survey}). 
Our first main result is the following  criterion ensuring the existence of an anti-torus. 

\begin{thm}\label{thm:anti-torus}
Let $T_1, T_2$ be locally finite trees with a cocompact automorphism group, and 
let $\Gamma \leq \Aut(T_1) \times \Aut(T_2)$ be a cocompact lattice.

Let also $(x_1, x_2) \in V(T_1) \times V(T_2)$ be a vertex, and  $\Gamma_1 = \Gamma_{x_1}$ and  $\Gamma_2 = \Gamma_{x_2}$ be the respective stabilizers of $x_1$ and $x_2$. Let $\gamma_1 \in \Gamma_1$ and $\gamma_2 \in \Gamma_2$ and let $\varphi \colon \Gamma \to G$ be a homomorphism to a group $G$ satisfying the following conditions:
\begin{enumerate}[label=(\roman*)]
\item $\varphi(\gamma_1)$ and $\varphi(\gamma_2)$ do not commute. 
\item For $i=1, 2$, the order of $\varphi(\gamma_i)$ is finite, and relatively prime to the order of the local action of $\Gamma$ at every vertex of $T_{i}$. 
\end{enumerate}
Then $T_1 \times T_2$ contains an anti-torus spanned by  powers of $\gamma_1$ and $\gamma_2$. In particular  $\Gamma$ is irreducible, and $\gamma_i$ acts as an automorphism of infinite order on $T_i$ for $i=1, 2$.
\end{thm}

Given a group $G$ acting on a tree $T$, the \textbf{local action} of $G$ at a vertex $v \in V(T)$ is the permutation group formed by the action of the stabilizer $G_v$ on the set of edges emanating from  $v$.  If $T$ is locally finite, this is  a finite group for every vertex. 

Following the terminology introduced in \cite{Cap_survey}, in the special case where the lattice $\Gamma$ acts regularly (i.e. sharply transitively) on the vertex set $V(T_1) \times V(T_2)$, we say that $\Gamma$ is a \textbf{BMW group}.  In that case, the $\Gamma$-action on both trees  $T_1$ and $T_2$ is vertex-transitive, so that both are homogeneous. The pair $(d_1, d_2)$ formed by the respective valencies of $T_1$ and $T_2$ is called the \textbf{degree} of $\Gamma$. Anti-tori have been mostly studied so far in the special case of BMW groups. Theorem~\ref{thm:anti-torus} is valid for arbitrary lattices in products of trees. 

It is an open question, recorded in \cite[\S4.4]{Cap_survey}, to determine whether every irreducible BMW group and, more generally, every irreducible lattice in a product of two trees, contains an anti-torus. Theorem~\ref{thm:anti-torus} allows us to answer that question in the case of $S$-arithmetic lattices in products of simple groups of rank~one:

\begin{cor}\label{cor:anti-tori:arithmetic}
Let $\Gamma$ is an $S$-arithmetic group in the product of two simple algebraic  groups of rank-one over local fields, whose respective Bruhat--Tits trees are denoted by $T_1$ and $T_2$. Then  $T_1 \times T_2$ contains an anti-torus. 
\end{cor}

%Indeed, congruence quotients of $\Gamma$ provide a sufficiently rich source of quotients. 

For special families of $S$-arithmetic BMW groups of quaternionic type, the existence of anti-tori has been established by D.~Rattaggi~\cite{Rattaggi_GeomDed} and Bondarenko--Bonda\-renko~\cite{BoBo24}.

As observed by D.~Wise \cite{Wise_PhD, Wise_CSC}, an anti-torus in a BMW group $\Gamma \leq \Aut(T_1) \times \Aut(T_2)$ gives rise to an aperiodic square  tiling of the plane: the square complex $T_1 \times T_2$ can be viewed as a presentation complex of $\Gamma$, whose squares are naturally equipped by a labelling by the defining relators of $\Gamma$. Each flat plane in $T_1 \times T_2$ inherits the structure of a square tiling colored by those relators. When the flat is an anti-torus, that tiling is aperiodic.  This is illustrated in Figure~\ref{figJW},   that represents an anti-torus found by Janzen--Wise  in a BMW group $\Gamma_{\mathrm{JW}}$ of degree~$(4, 4)$. 
As we shall see in Section~\ref{sec:anti-tori}, Theorem~\ref{thm:anti-torus} allows one to  recover that anti-torus (see Proposition~\ref{prop:JW1}), and to  answer positively a question asked in \cite[Section~2]{JW} concerning another lattice in a product of two tetravalent trees (see Proposition~\ref{prop:JW2}).

\begin{figure}[h]
\includegraphics[width=6.8cm]{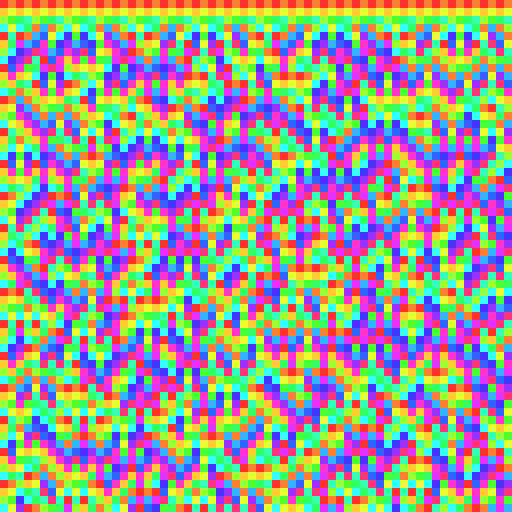} \hspace{.3cm}\includegraphics[width=6.8cm]{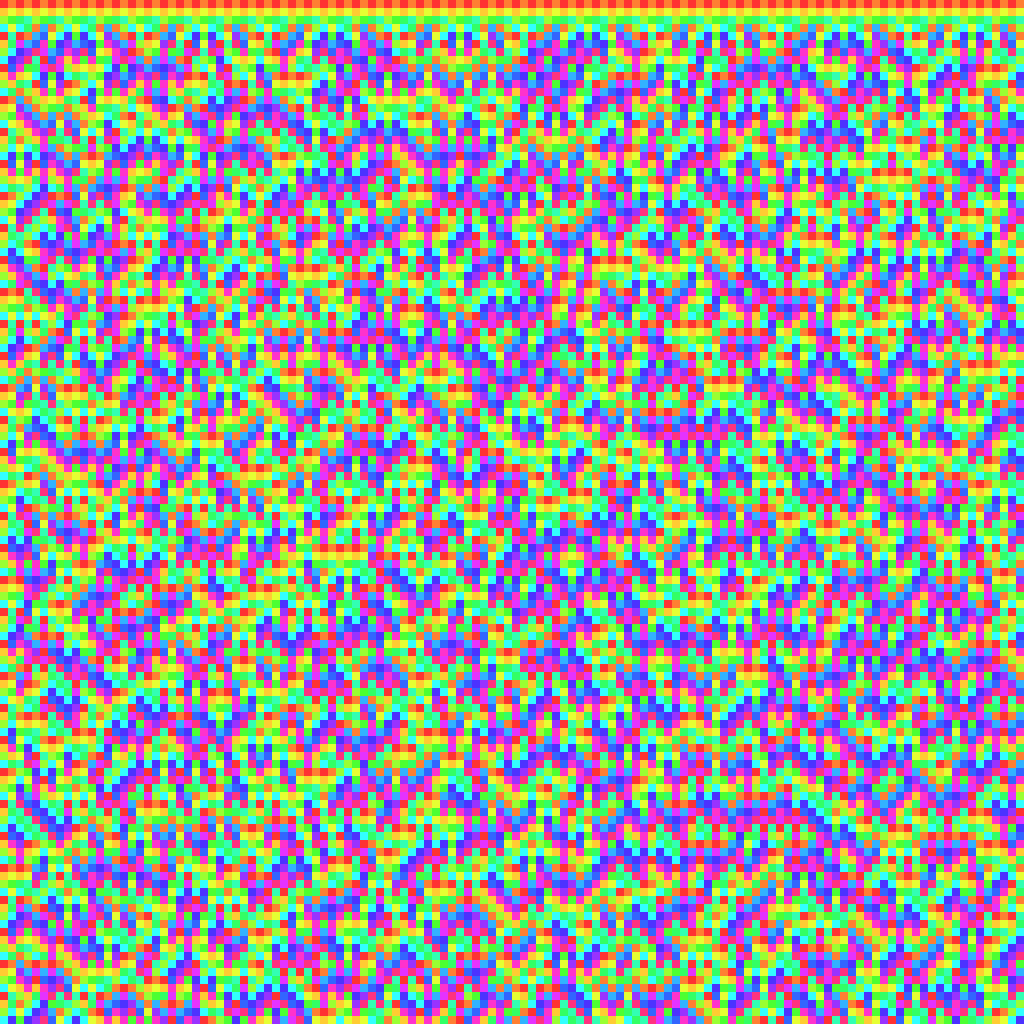}

\vspace{.3cm}
\includegraphics[width=6.8cm]{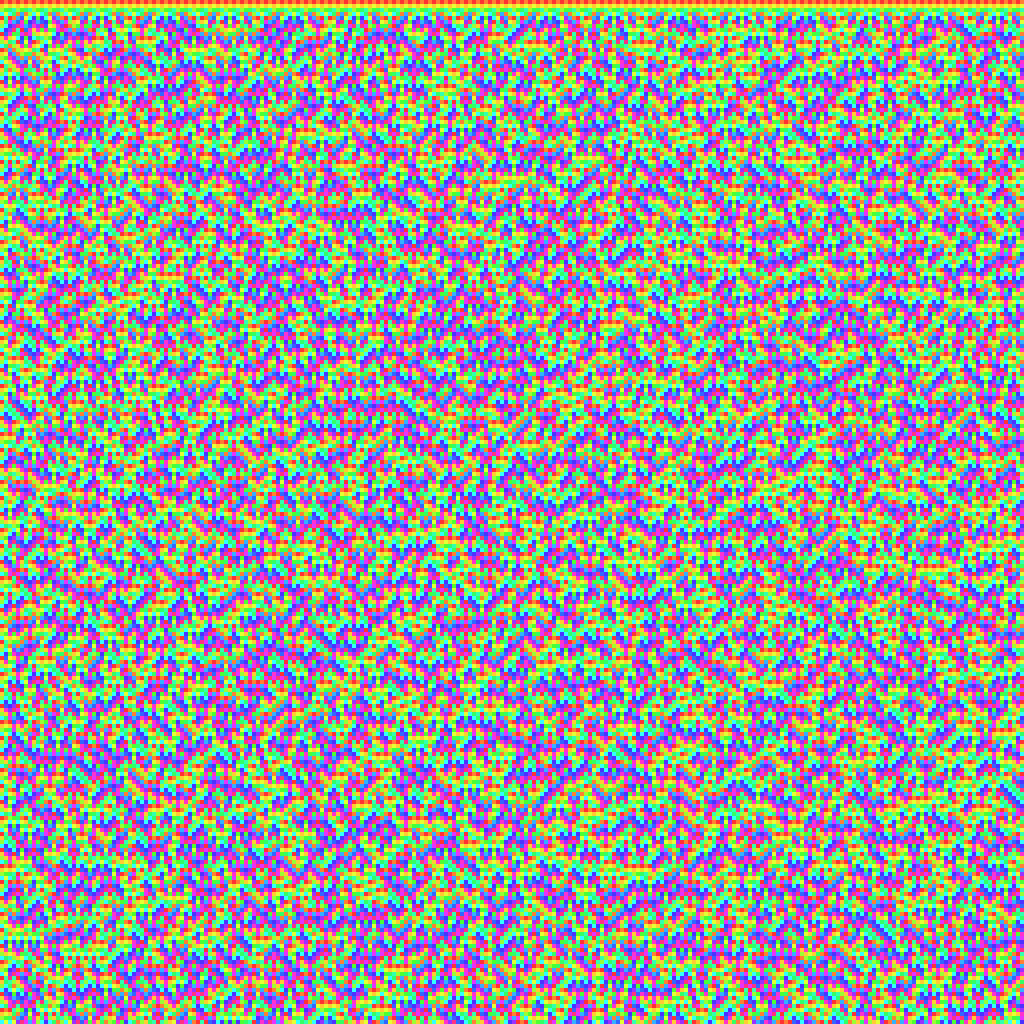} \hspace{.3cm}\includegraphics[width=6.8cm]{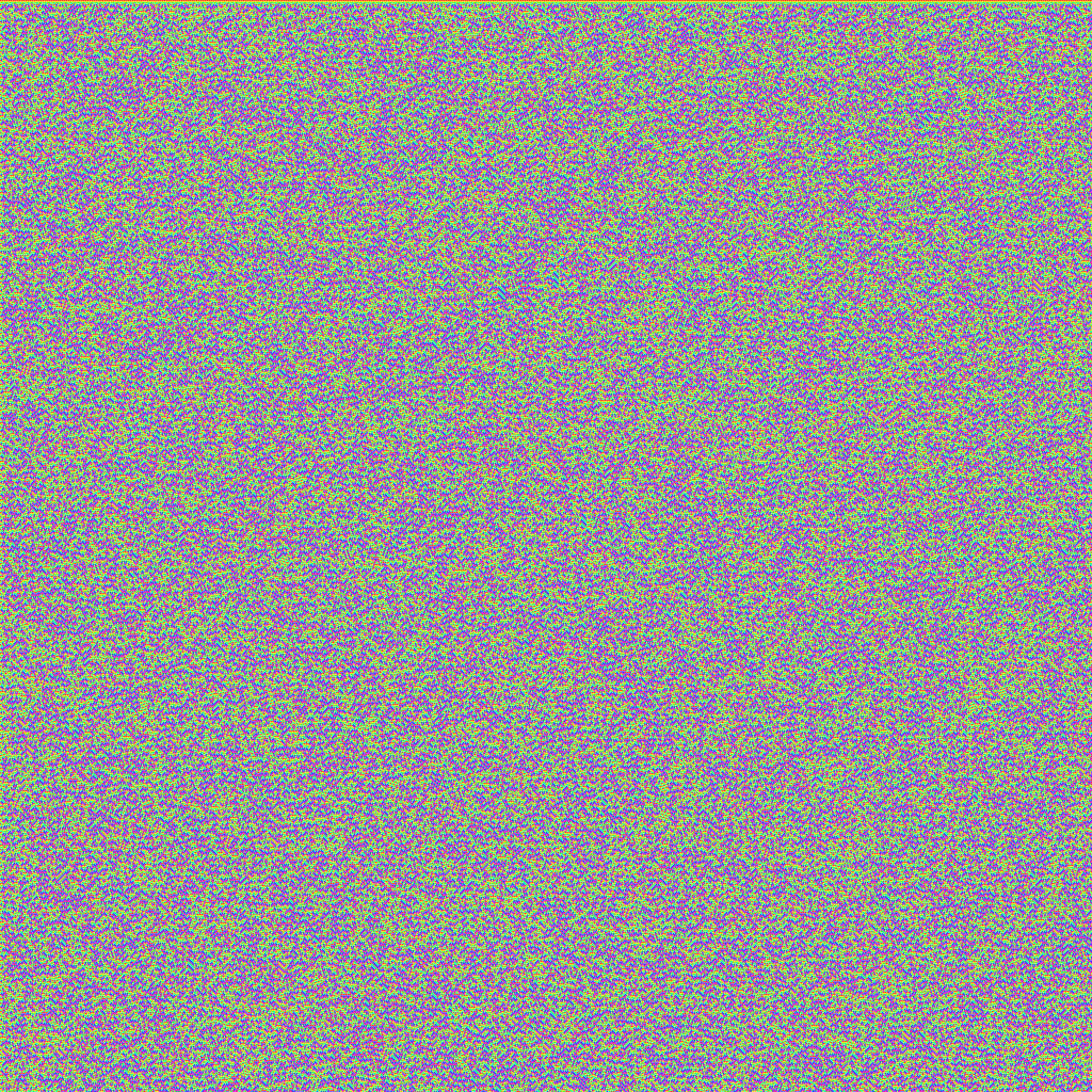}

\caption{The anti-torus of Janzen--Wise~\cite{JW}; the pictures represent a portion of size $2^n$   of the first quadrant, for  $n=6, 7, 8, 10$; the origin is placed on the top left corner and the $y$-axis is oriented downward. We have used $16$ different colors, each  representing a relation, taking into account the orientation of the corresponding square in the tiling.}
\label{figJW}
\centering
\end{figure}

A close connection between BMW groups and bi-reversible Mealy automata has been highlighted by Glasner--Mozes~\cite{GM}. The basic definitions concerning automata are recalled in Section~\ref{sec:bireversible} below; we refer to \cite{Nekra} for a detailed overview. Given a bi-reversible automaton $\mathcal B$, two basic questions are whether the automaton group $G_{\mathcal B}$ defined  by $\mathcal B$ is finite, and  whether it contains an element of infinite order. It has been conjectured in \cite{GodKli} that   $G_{\mathcal B}$ is infinite if and only if it contains an element of infinite order. Francoeur--Mitrofanov have shown that if  $G_{\mathcal B}$ contains an element of infinite order, then it contains a non-abelian free semigroup, see \cite[Corollary~1.2]{FM}. In this direction, we point out the following consequence of Theorem~\ref{thm:anti-torus}, providing a new criterion ensuring that $G_{\mathcal B}$ contains elements of infinite order. 

\begin{cor}\label{cor:birev-non-torsion}
Let $\mathcal B= (Q, L, \lambda, \rho)$ be a bi-reversible automaton. Let $n_{\mathcal B}$   be the order of the permutation group $\Sigma_{\mathcal B}\leq \Sym(L)$ generated by the invertible automaton $\mathcal B$, i.e. $\Sigma_{\mathcal B} = \langle \lambda(q, \cdot) \mid q \in Q\rangle$.  Let also $\Gamma_{\mathcal B}$ be the group defined by the presentation
$$\Gamma_{\mathcal B} = \langle  Q \sqcup L \mid q\lambda(q, \ell) = \ell \rho(q, \ell) \ \forall (q, \ell ) \in Q \times L\rangle.$$  
Suppose that   $\gamma_Q, \gamma_L \in \Gamma_{\mathcal B}$ are elements that belong respectively to the subgroups of $\Gamma_{\mathcal B}$ generated by $Q$ and $L$, and that  $\varphi\colon \Gamma_{\mathcal B} \to H$ is a homomorphism to a group $H$ satisfying the following conditions:
\begin{enumerate}[label=(\roman*)]
\item   $\varphi(\gamma_Q)$ ad $\varphi(\gamma_L)$ do not commute.
\item The order of $\varphi(\gamma_Q)$ is finite, and relatively prime to $n_{\mathcal B} n_{((\mathcal B^*)^{-1})^*}$. 
\item The order of $\varphi(\gamma_L)$ is finite, and relatively prime to $n_{\mathcal B^*} n_{(\mathcal B^{-1})^*}$. 
\end{enumerate}
Then the groups $G_{\mathcal B}$ and $G_{\mathcal B^*}$ defined by the automaton $\mathcal B$ and its dual $\mathcal B^*$ both contain non-abelian free semigroups. In particular they have exponential growth, and  contain elements of infinite order.
\end{cor}

We refer to Examples~\ref{ex:KPS1} and~\ref{ex:KPS2} for two illustrations concerning the  bi-reversible automata $\mathcal B$ and $\mathcal C$ studied in \cite[\S3.1 and \S3.2]{KPS}. According to the authors of \cite{KPS}, those examples stand out  among the non-symmetric $4$-state invertible automata over a $2$-letter alphabet for which standard methods of finding elements of infinite order do not work. We shall see that Corollary~\ref{cor:birev-non-torsion} applies to them, which recovers \cite[Propositions~3.1 and 3.6]{KPS} by  an alternative approach.

\medskip

In looking for other potential applications of Theorem~\ref{thm:anti-torus}, we explored the list of \textit{possibly irreducible} BMW groups of small degree that were enumerated by N.~Radu \cite{Radu_PhD, Radu}. Besides anti-tori, various other irreducibility criteria are known, notably in case the local action on one of the tree factor is doubly transitive, see \cite[Section~4]{Cap_survey}. In our exploration, we observed that Theorem~\ref{thm:anti-torus} applies to the following example, whose local actions on  both $T_1$ and $T_2$ are  $2$-groups (hence cannot be doubly transitive), so that   none of the known irreducibility criteria discussed in  \cite[Section~4]{Cap_survey} applies:

\begin{cor}\label{cor:first-fractal-example}
The group 
\begin{align*}
\Gamma_0 = \langle a, b, x, y, z, t \mid &x^2, y^2, z^2, t^2, \\
& axb^{-1} z, aya^{-1}t,  aza^{-1}x,  ata^{-1}y,
 byb^{-1} x,   bzb^{-1} t,   btb^{-1} y\rangle.  
\end{align*}
is an irreducible BMW group of degree $(4, 4)$. The elements   $a$ and $xy$ span an anti-torus.
\end{cor}

When we constructed the aperiodic tiling afforded by the anti-torus from Corollary~\ref{cor:first-fractal-example}, we were surprised to discover the fractal tiling depicted in Figure~\ref{fig1}, and the striking contrast with the anti-torus appearing in Figure~\ref{figJW}.

\begin{figure}[h]
\includegraphics[width=6.8cm]{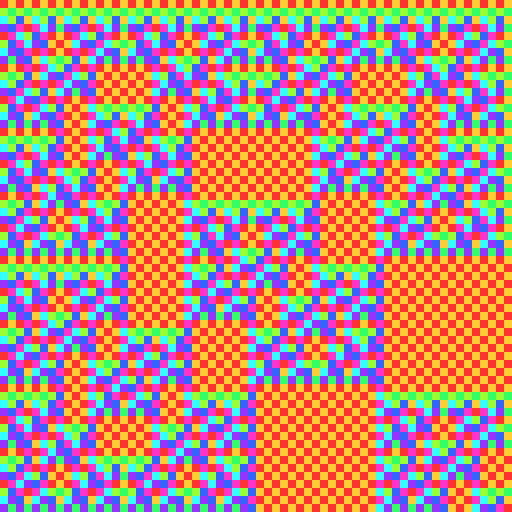} \hspace{.3cm}\includegraphics[width=6.8cm]{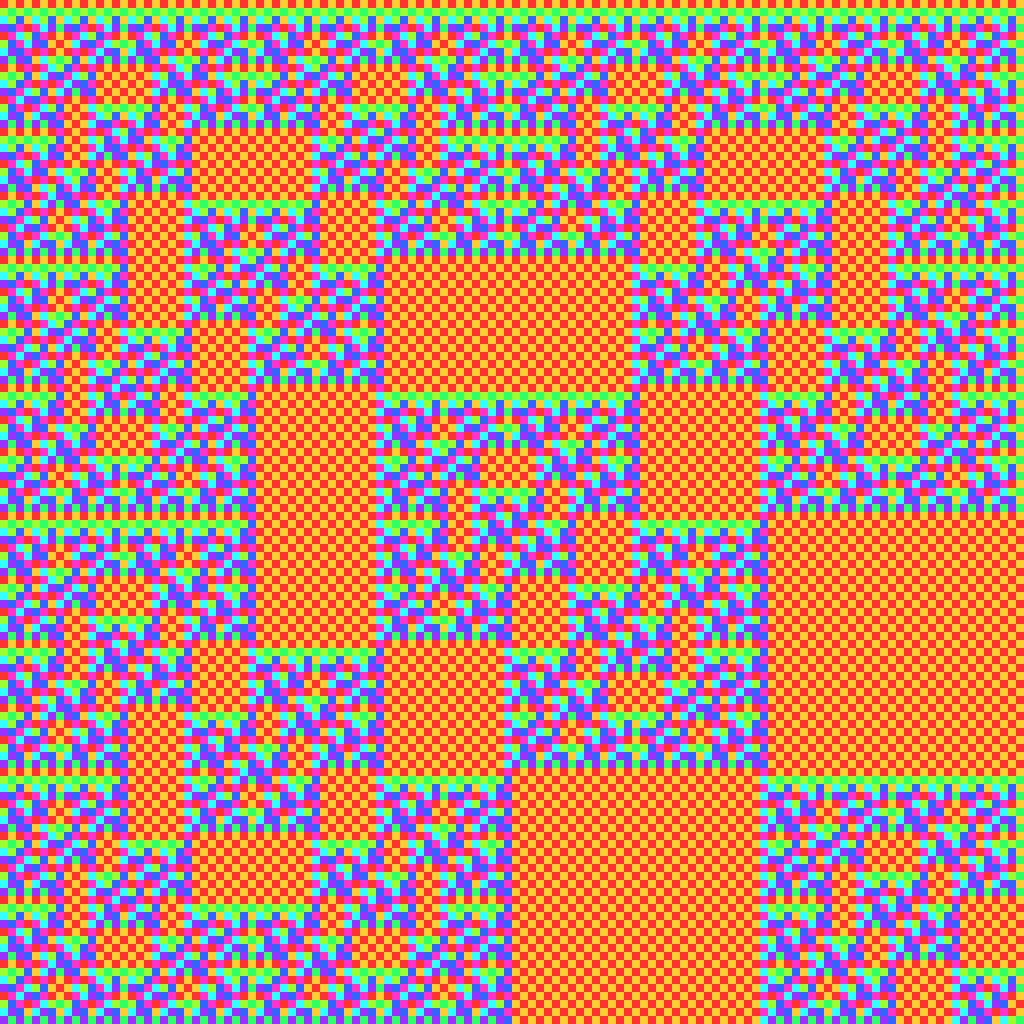}

\vspace{.3cm}
\includegraphics[width=6.8cm]{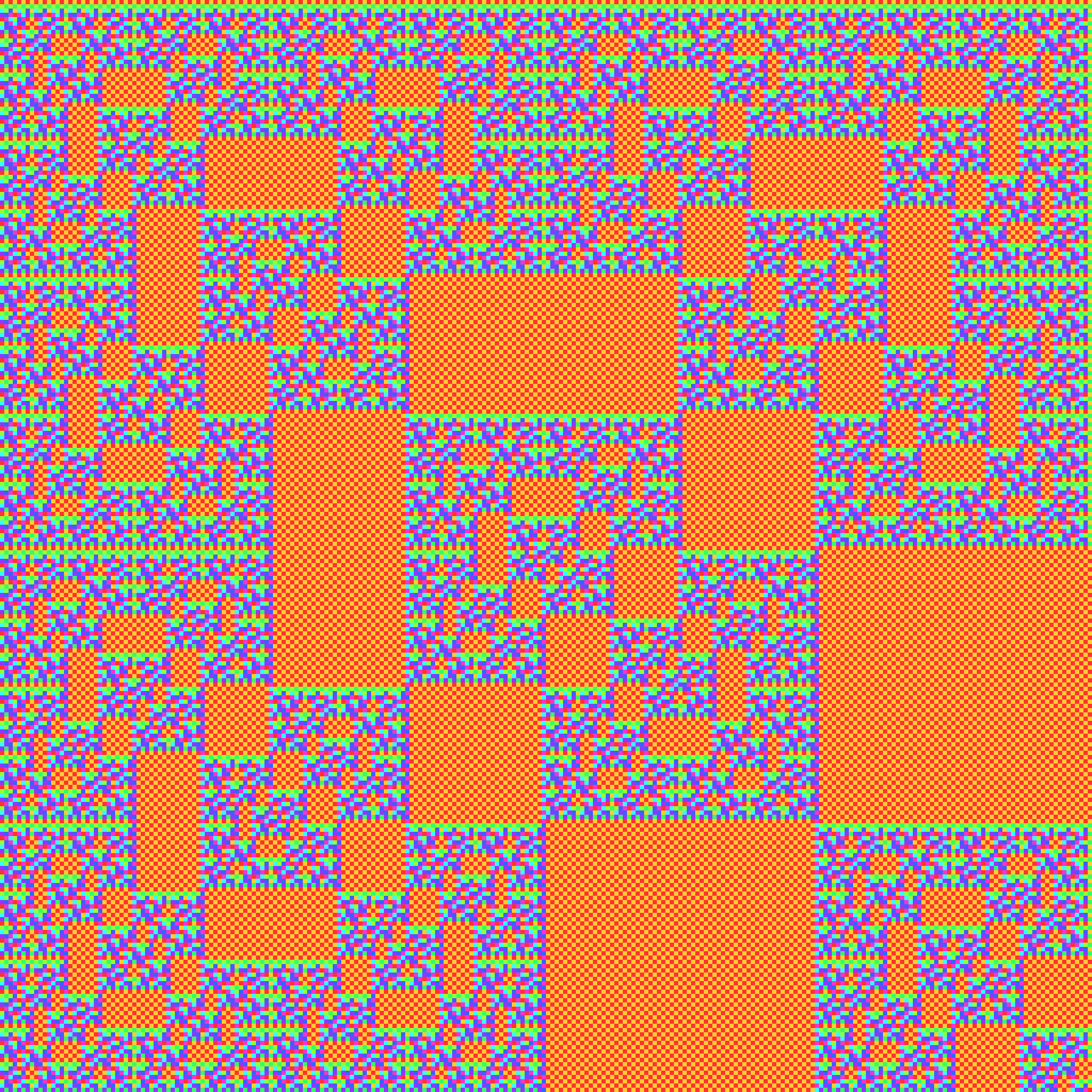} \hspace{.3cm}\includegraphics[width=6.8cm]{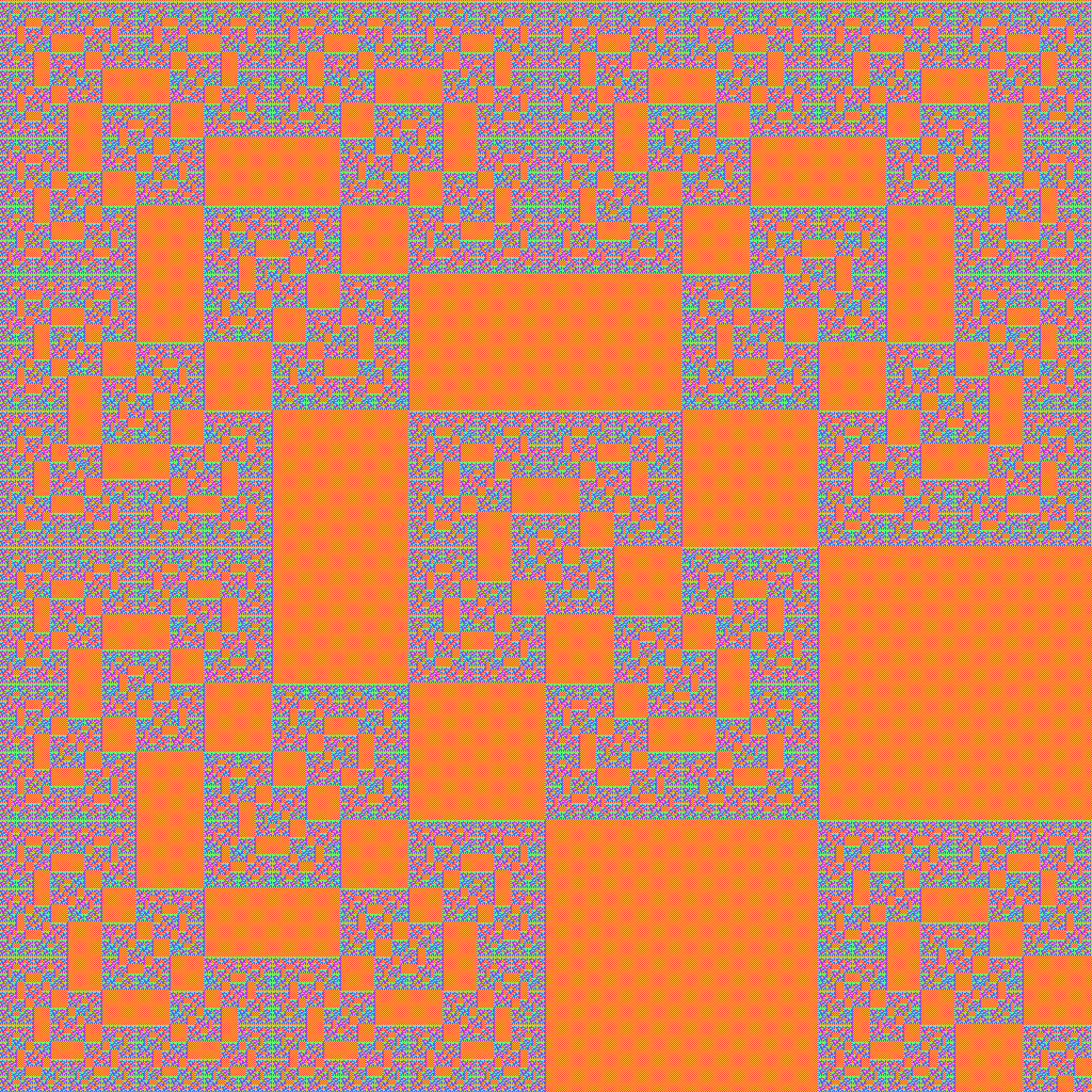}

\caption{The anti-torus afforded by Corollary~\ref{cor:first-fractal-example}, with the same conventions as in Figure~\ref{figJW}.}
\label{fig1}
\centering
\end{figure}

Our study of that example led us to several discoveries, some of which are presented in \cite{CV_chains}. In the rest of this introduction, we focus on another aspect, namely the description of  a new infinite family of irreducible BMW groups admitting fractal anti-tori. This family is indexed by a tuple  $\mathcal A = (V_1, V_2, a, b, c, d)$ consisting of a pair $V_1, V_2$ of finite-dimensional vector spaces over the same finite field $F$, together with  linear maps $a \colon V_1 \to V_1$, $b \colon V_2 \to V_1$, $c\colon V_1 \to V_2$ and $d \colon V_2 \to V_2$. To the tuple $\mathcal A = (V_1, V_2, a, b, c, d)$, we associate a finitely presented group $\Gamma_{\mathcal A}$ defined as follows:
$$\Gamma_{\mathcal A} = \langle V_1 \sqcup V_2 \mid v_1\cdot (cv_1 + d v_2) = v_2 \cdot (av_1 + bv_2) \  \forall (v_1, v_2)  \in V_1\oplus  V_2\rangle,$$
where the   symbol $\cdot$ denotes the group multiplication in $\Gamma_{\mathcal A}$. 
We warn the reader that we use the same symbol for vectors in the vector spaces $V_1$ and $V_2$, and for the generators of $\Gamma_{\mathcal A}$. This makes the notation lighter although it is a slight abuse (note that the zero vectors in $V_1$ and in $V_2$ correspond respectively to two different elements of infinite order in $\Gamma_{\mathcal A}$). It should not cause any confusion. 

We use the matrix notation $\left(\begin{matrix} a & b \\ c & d \end{matrix}\right)$ to denote the linear operator on  $V_1 \oplus V_2$ mapping $(v_1, v_2) $ to $(av_1+bv_2, cv_1+d v_2)$.
We shall assume throughout that the linear maps $a$, $d$ and $\left(\begin{matrix} a & b \\ c & d \end{matrix}\right)$ are invertible. This ensures that    $\Gamma_{\mathcal A}$ is a BMW group of degree $(2|V_1|, 2|V_2|)$, see Section~\ref{sec:bireversible}. We let $T_1, T_2$ be the associated trees. As mentioned above, the $2$-cells in the square complex $T_1 \times T_2$   are labelled by the defining relators of $\Gamma_{\mathcal A}$, which are naturally in bijection with the direct sum $V_1 \oplus V_2$. 
The key feature of   $\Gamma_{\mathcal A}$ is that, under a suitable hypothesis on $a, b, c, d$, it contains fractal anti-tori (see Theorem~\ref{thm:main} below). In order to specify the appropriate notion of fractal arising here, we need  the following terminology. 

Given a finite set $V$, a map $\mathcal T \colon \ZZ^2 \to V$ is viewed as a tiling of the Euclidean plane by unit squares, where each square is colored by an element of $V$ according to the map $\mathcal T$.  We say that $\mathcal T$ is a tiling with \textbf{color set} $V$. Given any integer $\ell \geq 1$, we may also view $\mathcal T$ as a tiling by squares of size $\ell$: the tiles are then colored by the elements of the function space $V^{\mathcal I^2}$ consisting of all maps $f \colon \mathcal I^2 \to V$, where $\mathcal I = \{0, 1, \dots, \ell-1\}$. We may view $\mathcal I^2$ as a square-shaped subset of $\ZZ^2$, so that each element of $V^{\mathcal I^2}$ defines a tiling of that square with $V$ as a color set. Formally, we define $\mathcal T^\ell \colon \ZZ^2 \to V^{\mathcal I^2}$ by setting 
$$\mathcal T^{\ell}(\balpha)  \colon \mathcal I^2 \to V : \bbeta \mapsto \mathcal T(\ell \balpha + \bbeta).$$
The tiling $\mathcal T^\ell$ is called the \textbf{tiling by blocks of size $\ell$} determined by $\mathcal T$. 

We say that $\mathcal T$ is  \textbf{invariant under a substitution} of length $\ell   \geq 2$   if there  is a map $\mathcal S \colon V \to V^{\mathcal I^2} : v \mapsto (\mathcal S_v \colon \mathcal I^2 \to V)$, such that 
$$\mathcal T(\ell \balpha + \bbeta) = \mathcal S_{\mathcal T(\balpha)}(\bbeta)$$
for all $\balpha \in \ZZ^2$ and $\bbeta \in \mathcal I^2$.  We say that $\mathcal T$ is  \textbf{generated by  a substitution} of length $\ell$ if there is a tiling $\overline{\mathcal T} \colon \ZZ^2 \to W$ with color set $W$ that is invariant under a substitution of length $\ell$, and a map $\tau \colon W \to V$, such that $\mathcal T = \tau \circ \overline{\mathcal T}$. The fact that tilings invariant under (or generated by) substitutions give rise to fractals is well known, see \cite{ShSt}.

\begin{thm}\label{thm:main}
If 
$$\det\left(\begin{matrix} 1-ax & bx \\ cy & 1-dy \end{matrix}\right) \neq \det(1-ax) \det(1-dy)$$ 
in the polynomial ring $F[x, y]$, then $\Gamma_{\mathcal A}$ is irreducible, and it contains a fractal anti-torus spanned by a pair $(v_1, v_2) \in V_1 \oplus V_2$. 

More precisely, the natural coloring of the anti-torus spanned by that pair $(v_1, v_2)$ forms a tiling $\mathcal T \colon \ZZ^2 \to V_1 \oplus V_2$ that  is generated by a substitution of length $p = \mathrm{char}(F)$.
Moreover there exist $s\geq 0$ and $t \geq 1$ such that the tiling $\mathcal T^{p^s}$ by blocks of size $p^s$ determined by $\mathcal T$ is invariant under a substitution of length $p^t$. 
\end{thm}

In other words, the presentation complex of $\Gamma_{\mathcal A}$, viewed as a  square complex whose squares are colored by the elements of $V_1 \oplus V_2$, contains flat subcomplexes that are fractal aperiodic tilings of the plane. 

The proof of Theorem~\ref{thm:main} relies on a construction of fractal tilings by matrices satisfying a recurrence relation, established in a separate paper \cite{CapVas} and recalled in Proposition~\ref{prop:autom-seq} below. Observe the similarity between the anti-torus depicted on Figure~\ref{fig1} and the fractal tilings shown in \cite[Figures~1, 2]{CapVas}. 

In the one-dimensional case, i.e. if $V_1 = V_2 = F$, the conditions imposed on $\mathcal A$ simply mean that the scalars $a, b, c, d$ and $ad-bc$ are all non-zero. The irreducibility of $\Gamma_{\mathcal A}$ in that special case may also be deduced from the work of Skipper--Steinberg~\cite{SkiSte} (see also \cite{BDR16} that corresponds to a special case over $F= \FF_3$), using the relation between BMW groups and bi-reversible automata established by Glasner--Mozes~\cite{GM} (see  Section~\ref{sec:bireversible} below). The existence of a (fractal)  anti-torus  is new.

In case $V_1 = F$ and $V_2 = F \oplus F$  and the maps $a, b, c, d$ are defined by $a(x) = x$, $b(y, z) = z$, $c(x) = (x, x)$ and $d(y, z) =(y, y+z)$,  the irreducibility of $\Gamma_{\mathcal A}$ may be deduced from the work of Francoeur~\cite{Fran}. Here again, the existence of an anti-torus  is new.

Taking $F = \FF_2$, $V_1 = V_2 = F^2$, $a = d = \left(\begin{matrix} 1& 1 \\ 0 & 1 \end{matrix}\right)$, $b = \left(\begin{matrix} 1& 0 \\ 0 & 1 \end{matrix}\right)$ and $c= \left(\begin{matrix} 1& 1 \\  1& 0 \end{matrix}\right)$, one obtains an irreducible BMW group that is closely related to the example from Corollary~\ref{cor:first-fractal-example}: indeed, it has a homomorphic image in the subgroup of $\Gamma_0$ generated by index two subgroups of $\langle a, b\rangle$ and $\langle x, y, z, t\rangle$ consisting of the elements of even length, see Remark~\ref{rem:Gamma0} below. In that way, Theorem~\ref{thm:main} provides an explanation for the fractal feature of the anti-torus depicted in Figure~\ref{fig1}. 

D.~Savchuk and T.~Smirnova-Nagnibeda  have  pointed out to us that the anti-torus depicted in Figure~2 looks similar to the picture appearing in \cite[Figure~9]{GriSav}, which is associated with the bi-reversible automaton $\mathcal C$ defined in \cite[Figure~7]{KPS}, and considered in Example~\ref{ex:KPS2} below. This is not a coincidence: indeed,  the BMW group $\Gamma_{\mathcal C}$ associated with $\mathcal C$ as in Corollary~\ref{cor:birev-non-torsion}  happens to have the BMW group $\Gamma_0$ from Corollary~\ref{cor:first-fractal-example} as a quotient. Hence, in view of the previous paragraph, Theorem~\ref{thm:main} also yields a conceptual explanation for the self-similarity of the picture in  \cite[Figure~9]{GriSav}.

\medskip
The fractal tilings afforded by Theorem~\ref{thm:main} are examples of \textit{$2$-dimensional automatic sequences}, as defined in \cite[Chapter~14]{AlSh}. Theorem~\ref{thm:main} may thus be viewed as providing a link between the bi-reversible automata defined by $\mathcal A$ and automatic sequences. Another connection between self-similar groups and automatic sequences has been highlighted in \cite{GLNS}. Many other links between self-similar groups and fractals have been identified and studied in the literature, see \cite{BGN, Nekra}. Note however that the automata mostly studied  in those references  in relation with fractals satisfy the condition of being \textit{contracting}, a condition that is somewhat orthogonal to bi-reversibility: indeed, as pointed to us by L.~Bartholdi and by T.~Smirnova-Nagnibeda, a group defined by a Mealy automaton that is both bi-reversible and contracting is necessarily finite\footnote{There does not seem to be any published proof of this   result at the time of this writing.}. 

%Our study of the example $\Gamma_0$ from Corollary~\ref{cor:first-fractal-example}, and the family $\Gamma_{\mathcal A}$ from Theorem~\ref{thm:main}, has led us to discover other interesting phenomena, see \cite{CV_chains}.  

\subsection*{Acknowledgement}

We thank Laurent Bartholdi for drawing our attention to the work of Skipper--Steinberg~\cite{SkiSte} and Francoeur~\cite{Fran}. We are grateful to L.~Bartholdi, M.~Burger, S.~Mozes, R.~Skipper and T.~Smirnova-Nagnibeda for their comments on a preliminary version of this article; we thank D. Savchuk and T.~Smirnova-Nagnibeda  for pointing out the similiarity between Figure~2 above and \cite[Figure~9]{GriSav}.

\section{Finding anti-tori via quotients}\label{sec:anti-tori}

We first record a basic  general fact.

\begin{prop}\label{prop:infinite-order}
For $i=1, 2$, let $G_i$ be a locally compact group and $U_i \leq G_i$  a compact open subgroup. Let $\Gamma \leq G_1 \times G_2$ be a discrete subgroup, let $\gamma_1 \in \Gamma \cap (G_1 \times U_2)$ and $\gamma_2 \in \Gamma \cap (U_1 \times G_2)$. 

Suppose that the projection of $\gamma_1$ to $G_2$ is of finite order. Then there exist positive integers $m_1, m_2 >0$ such that $\gamma_1^{m_1}$ and $\gamma_2^{m_2}$ commute. 
\end{prop}
\begin{proof}
Let $p_i \colon G_1 \times G_2 \to G_i$ be the canonical projection. Upon replacing $G_i$ by $\overline{p_i(\Gamma)}$, we may and shall assume that $\Gamma$ has dense projections. 

By hypothesis, there exists $m_1>0$ such that $\gamma_1^{m_1} \in \ker(p_2)$. Set $N = \Gamma \cap \ker(p_2)$. Since $N$ is normal in $\Gamma$, its projection $p_1(N)$ is normalized by $p_1(\Gamma)$. Since $\Gamma$ is discrete in $G_1 \times G_2$, it follows that $N$ is discrete. Since $N$ is contained in $p_1(N) \times \{e\}$, it follows that $p_1(N)$ is a discrete subgroup of $G_1$. In particular the normalizer $N_{G_1}(p_1(N))$ is closed, hence it contains $\overline{p_1(\Gamma)}=G_1$. This shows that $p_1(N)$ is a discrete normal subgroup of $G_1$. Thus the conjugacy class of $p_1(\gamma_1^{m_1})$ in $G_1$ is discrete. In particular, the $U_1$-conjugacy class of $p_1(\gamma_1^{m_1})$ is finite. It follows that $U_1$ has an open subgroup of finite index, say $U'_1$, that commutes with $p_1(\gamma_1^{m_1})$. Since $p_1(\gamma_2) \in U_1$,  there exists $m_2 >0$ such that $p_1(\gamma_2)^{m_2} \in U'_1$. It follows that $p_1(\gamma_1^{m_1})$ commutes with $p_1(\gamma_2^{m_2})$. Since $p_2(\gamma_1^{m_1} )$ is trivial, it commutes with $p_2(\gamma_2^{m_2})$. We infer that $\gamma_1^{m_1}$ and $\gamma_2^{m_2}$ commute.
\end{proof}

Our proof of Theorem~\ref{thm:anti-torus} is based on the following criterion. 

\begin{prop}\label{prop:anti-torus}
Let $T_1, T_2$, $\Gamma$, $x_1, x_2$ and $\gamma_1, \gamma_2$ be as in the notation of Theorem~\ref{thm:anti-torus}.  Let also $n_i$ be the least common multiple of the orders of the local actions of $\Gamma$ at the vertices of $T_i$, for $i=1, 2$. The following conditions are equivalent. 
\begin{enumerate}[label=(\arabic*)]
\item For all $m_1, m_2 > 0$, the elements $\gamma_1^{m_1}$ and $\gamma_2^{m_2}$ do not commute.

\item For all $k\geq 0$, the elements $\gamma_1^{n_1^k}$ and $\gamma_2^{n_2^k}$ do not commute.
\end{enumerate}
	
Moreover, if those conditions hold, then $T_1 \times T_2$ contains an anti-torus spanned by powers of $\gamma_1$ and $\gamma_2$. In particular  $\Gamma$ is irreducible, and $\gamma_i$ acts as an automorphism of infinite order on $T_i$ for $i=1, 2$.
\end{prop}

\begin{proof}
It is clear that (1) implies (2).  
%
%We assume henceforth that (1) fails, say $\gamma_1^{m_1}$ and $\gamma_2^{m_2}$ commute for some $m_1, m_2>0$. We shall prove by contradiction that (2) also fails. 
%
Let us now assume that (2) holds.  %Let $n_1$ (resp. $n_2$) be the least common multiple of the orders of the local actions of $\Gamma$ at the vertices of $T_2$ (resp. $T_1$). Hence, we assume that for all $k\geq 0$ and all positive integers $d_1,d_2$, respectively dividing $n_1^k$ and $n_2^k$, the elements  $\gamma_1^{d_1}$ and $\gamma_2^{d_2}$  do not commute.

Observe that for all $k \geq 0$, the element $\gamma_i^{n_i^k}$ acts trivially on the $k$-ball around $x_{i}$ in $T_{i}$. 
Conversely, for each $k \geq 0$, the order of the action of the element $\gamma_i$    on the $k$-ball around $x_{i}$ in $T_{i}$   divides $n_i^k$.  

We claim that $\gamma_i$ acts as a hyperbolic isometry on $T_{3-i}$, for $i=1, 2$. Indeed, suppose for a contradiction that $\gamma_1$ is not a hyperbolic isometry of $T_2$. Then it fixes some point $y_2 \in T_2$. By hypothesis $\gamma_1$ fixes $x_1 \in T_1$. Since the $\Gamma$-action on $T_1 \times T_2$ is proper,  the double stabilizer $\Gamma_{x_1, y_2}$ is finite. It follows that $\gamma_1$ is of finite order; in particular the order of the projection of $\gamma_1$ to $T_1$ is finite. By the previous paragraph, this implies that  there is $k \geq 0$ such that $\gamma_1^{n_1^k}$ acts trivially on $T_1$. In particular, for any integer $m$, the commutator $[\gamma_1^{n_1^k}, \gamma_2^m]$ acts trivially on $T_1$. 

Let $K_{x_1, y_2} \leq \Gamma_{x_1, y_2}$ be the kernel of the $\Gamma_{x_1, y_2}$-action  on $T_1$. We have just seen that $[\gamma_1^{n_1^k}, \gamma_2^m] \in K_{x_1, y_2} $ for all $m$. Moreover, by definition $K_{x_1, y_2}$ is a finite group that acts faithfully on $T_2$. Hence  $K_{x_1, y_2}$ acts faithfully on the $r$-ball around $y_2$ in $T_2$, for some sufficiently large $r>0$. Let $t = d(x_2, y_2)$ be the distance from $x_2$ to $y_2$ in $T_2$. Since $\gamma_2$ fixes $x_2$, we  may assume, upon enlarging $k$ if necessary,  that $\gamma_2^{n_2^k}$ fixes pointwise the $(r+t)$-ball around $x_2$. In particular $\gamma_2^{n_2^k}$ fixes pointwise the $r$-ball around $y_2$. It follows that the commutator  $[\gamma_1^{n_1^k}, \gamma_2^{n_2^k}]  \in K_{x_1, y_2}$ acts trivially on the $r$-ball around $y_2$, hence trivially on $T_2$ by the choice of $r$. We infer that $[\gamma_1^{n_1^k}, \gamma_2^{n_2^k}]$ is trivial, which contradicts (2). This proves that $\gamma_1$ acts as a hyperbolic isometry on $T_2$. A similar argument shows that  $\gamma_2$ acts as a hyperbolic isometry of $T_1$. The claim stands proven.

For $i=1, 2$, we let $\ell_{3-i} \subset T_{3-i}$ be the unique translation axis of $\gamma_i$ in the tree $T_{3-i}$. Thus the translation axes of $\gamma_1$ in the product $T_1 \times T_2$ are of the form $\{y_1\} \times \ell_2$, where $y_1 \in T_1$ is a fixed point  of $\gamma_1$ (and similarly for $\gamma_2$). 

For each $j \geq 0$, the pair  $(\gamma_1^{n_1^j}, \gamma_2^{n_2^j})$ also satisfies (2).  
%Moreover, by our assumption that (1) fails, we know that $\gamma_1^{m_1}$ and $\gamma_2^{m_2}$ commute, so that $\gamma_1^{n_1^km_1}$ and $\gamma_2^{n_2^k m_2}$ also commute. 
Therefore, there is no loss of generality in replacing $\gamma_i$ by $\gamma_i^{n_i^j}$. Upon  choosing $j$ large enough, we may  assume that $\gamma_1$ fixes a vertex $x'_1 \in \ell_1$ and $\gamma_2$ fixes a vertex $x'_2 \in \ell_2$. This ensures that $\gamma_1$ has a translation axis of the form $\{x'_1\} \times \ell_2$   and $\gamma_2$  has a translation axis of the form $\ell_1 \times \{x'_2\}$, with $x'_i \in \ell_i$ for $i=1, 2$. Let $P \subset T_1 \times T_2$ be the flat spanned by the geodesic lines $\ell_1 \times \{x'_2\}$ and $ \{x'_1\} \times \ell_2 $. Let also $L =\mathrm{Stab}_\Gamma(P)$. 

\medskip
Observe that, by the Flat Torus Theorem (see \cite[Chapter II.7]{BH}), if $\gamma_1^{m_1}$ and $\gamma_2^{m_2}$ commuted for some $m_1, m_2>0$, then $P$ would be periodic.

\medskip 
We next claim that   $P$   is not periodic, so that it forms an anti-torus. By  the previous paragraph, this implies that (1) holds, as required. 
Suppose for a contradiction that $L$ acts cocompactly on $P$. By Bieberbach's theorem, it follows that $L$ has a finite index subgroup $L_0$  acting (cocompactly) by translations on $P$. The kernel of the $L_0$-action on $P$ is a finite normal subgroup $K$. (Thus $L_0$ is virtually abelian, but it need not be abelian a priori.) Since $K$ fixes the vertex $(x'_1, x'_2)$, it acts on the balls around that vertex in $T_1 \times T_2$. Since $K$ is finite, it follows for any  sufficiently large $r >0$, the $K$-action on the $r$-ball around $(x'_1, x'_2)$ is faithful.  

We next claim that some positive power of $\gamma_1$ (resp. $\gamma_2$) belongs to $L_0$. Indeed, the $L_0$-action on $P$ preserves the square tiling induced by the graph structure on $T_1 \times T_2$. This implies that $L_0$ contains a non-trivial element $t_1$ acting as a translation with axis $\{x'_1\} \times \ell_2$. Considering the least common multiple of the respective translation lengths of $\gamma_1$ and $t_1$, we find infinitely many distinct pairs of non-zero integers $(k_i, l_i)$ such that $\gamma_1^{k_i} t_1^{l_i}$ fixes the vertex $(x'_1, x'_2)$. Since the $\Gamma$-action is proper, the vertex-stabilizers are finite. By the pigeonhole principle, we may find $i, j$ with $k_i < k_j$ and $\gamma_1^{k_i} t_1^{l_i} = \gamma_1^{k_j} t_1^{l_j}$. It follows that $\gamma_1^{k_j-k_i}$ is a positive power of $\gamma_1$ that belongs to $L_0$, as claimed. The argument for $\gamma_2$ is similar. 

By construction $\gamma_1$ preserves the lines $\ell_2 \subset T_2$. The previous claim implies moreover that the orbit of the line $\ell_1 \subset T_1$ under the cyclic group $\langle \gamma_1 \rangle$ is finite. By the orbit counting formula, the size of that orbit divides $n_1^k$ for some $k$. It follows that there exists $k \geq 0$ such that $\gamma_1^{n_1^k}$ fixes $\ell_1$ pointwise. Hence for each $m \geq 0$, there exists an integer $k(m)$ such that  $\gamma_1^{n_1^{k(m)}}$ fixes pointwise the $m$-neighbourhood of $\ell_1$ in $T_1$. Similarly, for each $m \geq 0$, there exists an integer $k'(m)$ such that $\gamma_2^{n_2^{k'(m)}}$ fixes pointwise the $m$-neighbourhood of $\ell_2$ in $T_2$. Choose $j > \max\{k(r), k'(r)\}$, where $r$ is as above. 

We infer that $\gamma_1^{n_1^j}$ and $\gamma_2^{n_2^j}$ both belongs to $L_0$. Moreover their  commutator 
$[\gamma_1^{n_1^j}, \gamma_2^{n_2^j}]$ is an element of $K$ that fixes pointwise the $r$-ball around $\ell_1$ in $T_1$, as well as the $r$-ball around $\ell_2$ in $T_2$. In particular $[\gamma_1^{n_1^j}, \gamma_2^{n_2^j}]$ fixes pointwise the $r$-ball around $(x'_1, x'_2) \in \ell_1 \times \ell_2$. Hence it is trivial by the definition or $r$. 

This proves that the elements $\gamma_1^{n_1^j}$ and $\gamma_2^{n_2^j}$ commute, contradicting the assumption that (2) holds. This confirms the claim that the flat $P$ is an anti-torus spanned by powers of $\gamma_1$ and $\gamma_2$, hence also that (1) holds as observed above. 

It follows from Proposition~\ref{prop:infinite-order} that $\gamma_i$ acts as an automorphism of infinite order on $T_i$. The irreducibility of $\Gamma$  follows from \cite[Prop.~4.11]{Cap_survey}.
\end{proof}

\begin{proof}[Proof of Theorem~\ref{thm:anti-torus}]
Let $n_i$ be as in Proposition~\ref{prop:anti-torus}. It follows from the hypothesis (ii) that for all $k \geq 0$, the cyclic groups generated by $\varphi(\gamma_i)$ and $\varphi(\gamma_i^{n_i^k})$ coincide. In view of the hypothesis (i), those cyclic groups do not commute. Hence for all $k \geq 0$, the elements $\gamma_1^{n_1^k}$ and $\gamma_2^{n_2^k}$ do not commute. Thus the required conclusions all follow from Proposition~\ref{prop:anti-torus}.
\end{proof}

\begin{proof}[Proof of Corollary~\ref{cor:anti-tori:arithmetic}]
Let $\Pi$ be the set of prime divisors of the orders of the local actions of $\Gamma$ as the vertices of $T_1$ and $T_2$. Since $\Gamma$ acts cocompactly, the set $\Pi$ is finite. 

Upon replacing $\Gamma$ be a finite index subgroup, we may assume that the first congruence quotients of $\Gamma$ are quasi-simple. Let $\varphi \colon \Gamma \to G$ be a congruence quotient of $\Gamma$ associated with a prime ideal $I$ of   index $q$ in the ring of $S$-integers of the global field on which $\Gamma$ is realized as an $S$-arithmetic group. Upon choosing the index large $q$ enough, we may assume, by strong approximation, that $\varphi(\Gamma_1) = G = \varphi(\Gamma_2)$. 

The group  $G$ is a quasi-simple group of Lie type over a finite field of order $q = p^e$ and characteristic $p$. We claim that, upon choosing $q$ large enough, there is a prime divisor $s$ of the order of $G$ that is not contained in $\Pi$. 

We choose $s$ be a \textit{primitive prime divisor}  of $q-1$, see \cite[Section~2.4]{LPS90}. This means that $s$ divides $q-1$ but not $p^{b}-1$ for any $b \leq e-1$. In particular, upon choosing a larger $q$ if needed, we may assume that $s$ does not belong to $\Pi$. (In case the ground field over which $\Gamma$ is defined has characteristic~$0$, we could have simply taken $s = p$ for a suitable chosen ideal $I$, since in that case $\Gamma$ admits congruence quotients over fields of arbitrarily large characteristic.) 

Let now $g \in G$ be an element of order $s$. Since $G$ is quasi-simple, some conjugate of $g$, say $h$, does not commute with $g$. 
Let $\gamma_1 \in \Gamma_1$ be such that $\varphi(\gamma_1) = g$, and $\gamma_2 \in \Gamma_2$ be such that $\varphi(\gamma_2) = h$. The hypotheses of Theorem~\ref{thm:anti-torus} are satisfied, thereby ensuring that $T_1 \times T_2$ contains an anti-torus. 
\end{proof}
	
Let us now illustrate Theorem~\ref{thm:anti-torus} on  the BMW group of degree $(4, 4)$ introduced and studied by Janzen--Wise~\cite{JW}. That groups   admits  the following presentation\footnote{We have used a different labelling of the edges of the defining square complex than in~\cite{JW}, in accordance with the choice made in \cite{Cap_survey}. The labelling from~\cite{JW} is obtained by replacing $(a, b, x, y)$ by $(b, a, x^{-1}, y)$.}:
$$
\Gamma_{\mathrm{JW}}=  \langle a, b, x, y \mid 
 axay, ax^{-1} b y^{-1}, 
 ay^{-1} b^{-1}x^{-1}, bxb^{-1}y^{-1}. \rangle 
$$
 Using Theorem~\ref{thm:anti-torus}, we recover the following alternative proof of the main result of~\cite{JW}. 

\begin{prop}\label{prop:JW1}
The elements $ab$ and $xy$ of $\Gamma_{\mathrm{JW}}$ span an anti-torus. 
\end{prop}
\begin{proof}	
One checks that the assignments
\begin{align*}
a & \mapsto (1, 5, 12, 2, 10, 6, 8, 3, 7, 4)\\
b & \mapsto (1, 5, 3, 11, 12, 8, 6, 10, 7, 9, 2, 4)\\
x& \mapsto (1, 8, 11)(2, 7, 5)(3, 12, 10)(4, 6, 9)\\
y & \mapsto (1, 9, 10)(2, 8, 7)(3, 4, 12)(5, 11, 6)
\end{align*}
extend to a group homomorphism  $\varphi \colon \Gamma_{\mathrm{JW}} \to \Sym(12)$ (whose image is isomorphic to $\mathrm{PGL}_2(11)$). The permutations  $\varphi(ab) = (1, 3, 9, 2, 7)(4, 5, 8, 11, 12)$ and $\varphi(xy)=(1, 7, 11, 9, 12)(4, 5, 8, 6, 10)$ are both of order~$5$, and they do not commute. On the other hand, since both trees $T_1$ and $T_2$ have valency~$4$, the orders of the local actions both divide $24$.  Hence Theorem~\ref{thm:anti-torus} ensures that translation axes of $ab$ and $xy$ span an anti-torus. 
\end{proof}

We next consider the group 
$$
\Gamma_{\mathrm{JW2}}=  \langle a, b, x, y \mid 
a^2 x^2, a b a y, x y^2 b, x b^2y \rangle 
$$
introduced in \cite[Section~2]{JW}\footnote{The group $\Gamma_{\mathrm{JW2}}$ is defined in \cite{JW} as the fundamental group of a square complex appearing in Figure~6 of \cite{JW}. That figure contains a misprint: the orientation of the arrow appearing at the bottom of the second square should be reversed.}. As the authors point out, the group $\Gamma_{\mathrm{JW2}}$ has a faithful geometric action on a product $T_1\times T_2$ of two tetravalent tree; that action is sharply transitive on the vertices, but it swaps the two factors. Thus $\Gamma_{\mathrm{JW2}}$ has a subgroup of index~$2$ that embeds as a cocompact lattice in $\Aut(T_1)\times \Aut(T_2)$. The following result answers a question left open at the end of Section~2 in \cite{JW}. 

\begin{prop}\label{prop:JW2}
The elements   $bx$ and $ya$ in $\Gamma_{\mathrm{JW2}}$  span an anti-torus.  
\end{prop}
\begin{proof}
Since both trees $T_1$ and $T_2$ have valency~$4$, the orders of the local actions both divide $24$. We let   $\Gamma$ be the index~$2$ subgroup of $\Gamma_{\mathrm{JW2}}$ that preserves both tree factors. Thus $\Gamma$  embeds as a cocompact lattice in $\Aut(T_1)\times \Aut(T_2)$. Moreover $bx$ and $ya$ both belong to $\Gamma$. One checks that the assignments
\begin{align*}
a & \mapsto (1, 2, 5, 11, 8, 3, 6, 13, 7, 4)(9, 17, 19, 15, 18, 16, 10, 14, 20, 12)\\
b & \mapsto (1, 3, 7, 15, 19, 11, 18, 14, 6, 2)(4, 9, 13, 5, 12, 8, 16, 20, 17, 10)\\
x & \mapsto (1, 14, 7, 16, 6, 15, 8, 17, 5, 12)(2, 10, 4, 18, 13, 19, 3, 9, 11, 20)\\
y & \mapsto (1, 16, 5, 3, 20, 15, 19, 14, 9, 2)(4, 8, 17, 7, 12, 18, 13, 10, 11, 6)
\end{align*}
extend to a group homomorphism  $\varphi \colon \Gamma \to \Sym(20)$ (whose image is an index~$2$ subgroup of $\mathrm{PGL}_2(9) \wr C_2$). We have 
\begin{align*}
\varphi(ya)^2 &= (1, 8, 20, 7, 5)(3, 16, 13, 17)(4, 12, 11, 14)(6, 10, 19, 18, 9)\\
\varphi(bx) &= (1, 9, 19, 20, 5)(2, 14, 15, 3, 16)(4, 11, 13, 12, 17)(6, 10, 18, 7, 8),
\end{align*} 
so that the  images of $bx$ and $(ya)^2$ both have order~$5$, and they do not commute. Thus Theorem~\ref{thm:anti-torus} applies. This ensures that $bx$ and $ya$ span an anti-torus in $\Gamma$, hence also in $\Gamma_{\mathrm{JW2}}$.
\end{proof}
		
\begin{remark}
It turns out that both $\Gamma_{\mathrm{JW}}$ and $\Gamma_{\mathrm{JW2}}$ admit a homomorphic image that is a quaternionic arithmetic lattice, respectively defined over the global fields $\mathbf Q$ and $\FF_3(t)$. This provides an alternative approach to produce finite quotients satisfying the hypotheses of Theorem~\ref{thm:anti-torus} and thus to construct anti-tori in those groups, by proceeding as in the proof of Corollary~\ref{cor:anti-tori:arithmetic}. We omit the details. 
\end{remark}

Finally, we indicate how to deduce Corollary~\ref{cor:first-fractal-example} from Theorem~\ref{thm:anti-torus}. 

\begin{proof}[Proof of Corollary~\ref{cor:first-fractal-example}]
The local actions of $\Gamma_0$ on both trees are $2$-groups (the local actions can be easily identified on the basis of the presentation of a BMW group, see \cite[Section~4.5]{Cap_survey}). One checks that the assignments 
\begin{align*}
a & \mapsto (1, 5, 2, 6, 3, 7)(4, 8)\\
b & \mapsto (1, 6, 2, 5, 3, 8)(4, 7)\\
x& \mapsto (1, 6)(2, 8)(3, 7)(4, 5)\\
y & \mapsto (1, 5)(2, 6)(3, 7)(4, 8)\\
z & \mapsto (1, 7)(2, 8)(3, 5)(4, 6)\\
t & \mapsto (1, 7)(2, 5)(3, 6)(4, 8)
\end{align*}
extend to a group homomorphism  $\varphi \colon \Gamma_0 \to \Sym(8)$ (whose image is isomorphic to  $\Alt(4) \wr C_2$). Clearly, $\varphi(a^2)= (1, 2, 3)(5,  6,7)$ and $\varphi(xy) = (1, 2, 4)(5, 8, 6)$ have order~$3$. Moreover they do not commute. Hence $a$ and $xy$ span an anti-torus by Theorem~\ref{thm:anti-torus}. 
\end{proof}

\section{Mealy automata and BMW groups}\label{sec:bireversible}

A \textbf{Mealy automaton} is a tuple $\mathcal B = (Q, L, \lambda, \rho)$ consisting of finite sets $Q$ and $L$ respectively called the \textbf{set of states} and the \textbf{alphabet}, and maps $\lambda  \colon Q \times L \to L$ and $\rho \colon Q \times L \to Q$ respectively called the \textbf{output map} and the \textbf{transition map}. The automaton $\mathcal B$ is called \textbf{invertible} if  for all $q \in Q$, the map  $ L \to L : \ell \mapsto \lambda(q, \ell)$ is a bijection. It is  called \textbf{reversible} if  for all $\ell \in L$ the map $Q \to Q : q \mapsto \rho(q, \ell)$ is a bijection. We say that $\mathcal B$ is   \textbf{bi-reversible} if it is both invertible and reversible, and if moreover the map $Q \times L \to  L\times Q : (q, \ell) \mapsto \big(\lambda(q, \ell), \rho(q, \ell) \big)$ is a bijection. 

To each Mealy automaton $\mathcal B  = (Q, L, \lambda, \rho)$, one associates the \textbf{dual automaton}, defined as  $\mathcal B^* = (L, Q, \rho \circ \tau, \lambda \circ \tau)$, where $\tau \colon L \times Q \to Q \times L$ is the map that swaps the coordinates. If $\mathcal B$ is invertible, one also defines the \textbf{inverse automaton} $\mathcal B^{-1}$ whose set of states is a set $Q^{-1}$ in bijection with $Q$, whose alphabet is $L$, and whose output maps are defined as the inverses of the output maps of $\mathcal B$ (see \cite[Section~1.3.8]{Nekra}). 
The condition that $\mathcal B$ be bi-reversible amounts to requiring that the automata $\mathcal B$, $\mathcal B^*$ and $(\mathcal B^{-1})^*
$ are all invertible (see \cite[Section~1.10]{Nekra}). 

If the Mealy automaton $\mathcal B$ is invertible, it defines a  group denoted by $G_{\mathcal B}$ that acts on the regular rooted tree, whose vertices are naturally in bijection with the free monoid generated by the alphabet $L$, see \cite[Def.~1.5.7]{Nekra}.

Glasner--Mozes \cite{GM} have shown that every bi-reversible automaton $\mathcal B$ gives rise, in a canonical way, to a BMW group, namely the group $\Gamma_{\mathcal B}$ considered in Corollary~\ref{cor:birev-non-torsion}. This allows us to deduce Corollary~\ref{cor:birev-non-torsion} from Theorem~\ref{thm:anti-torus}. 

\begin{proof}[Proof of Corollary~\ref{cor:birev-non-torsion}]
By \cite[Proposition~2.6]{GM}, the group $\Gamma_{\mathcal B}$ is a BMW group, acting regularly on the vertices of a product $T_Q \times T_L$ of two trees. The tree $T_Q$ (resp. $T_L$) is the Cayley tree of the subgroup of  $\Gamma_{\mathcal B}$ generated by $Q$ (resp. $L$). The local action of $\Gamma_{\mathcal B}$  on $T_Q$ is a subgroup of the direct product $\Sigma_{\mathcal B^*} \times \Sigma_{(\mathcal B^{-1})^*}$, where $\Sigma_{\mathcal B^*} \leq \Sym(Q)$ is generated by $\{ \rho(\cdot , \ell) \mid \ell \in L\}$, and $ \Sigma_{(\mathcal B^{-1})^*}$ is the corresponding permutation group if $\mathcal B$ is replaced by $\mathcal B^{-1}$. In particular, the order of the local action of $\Gamma_{\mathcal B}$  on $T_Q$ is a divisor of $n_{\mathcal B^*}  n_{(\mathcal B^{-1})^*}$. Similarly, the order of the local action of $\Gamma_{\mathcal B}$  on $T_L$ divides $n_{\mathcal B} n_{((\mathcal B^*)^{-1})^*}$. Hence, under the assumptions of Corollary~\ref{cor:birev-non-torsion}, we see that the hypotheses of Theorem~\ref{thm:anti-torus} are fulfilled. It follows that $\gamma_Q$ (resp. $\gamma_L$) acts as an element of infinite order on $T_L$ (resp. $T_Q$). In view of \cite[Proposition~2.12]{GM}, this implies that the automaton groups $G_{\mathcal B}$ and $G_{\mathcal B^*}$ both contain elements of infinite order. Now, the required conclusion follows from \cite[Corollary~1.2]{FM}. 
\end{proof}

In order to illustrate Corollary~\ref{cor:birev-non-torsion}, we consider two specific bi-reversible automata, denoted $\mathcal B$ and $\mathcal C$, studied in \cite[\S 3]{KPS}.

\begin{example}\label{ex:KPS1}
The automaton $\mathcal B$ with set of states $Q = \{q, w, e, r\}$ and alphabet $L = \{0, 1\}$  is defined on  Figure~5 of \cite{KPS}. We do not reproduce that figure, but it can be reconstructed from the presentation of the associated BMW group $\Gamma_{\mathcal B}$, which is the following:
\begin{align*}
\Gamma_{\mathcal B} = \langle q, w, e, r, 0, 1| & 
q0=1r, \  q1=0w, \ w0=0q, \ w1=1q, \\ 
&e0=1w, \ e1=0r, \ r0=1e, \ r1=0e\rangle.
\end{align*}
One checks that $n_{\mathcal B*} = n_{(\mathcal B^{-1})^*} = 8$, and $n_{\mathcal B} = n_{((\mathcal B^*)^{-1})^*} = 2$. Moreover, the assignments
\begin{align*}
q & \mapsto (3, 8)(4, 12, 11, 5)(9, 10)\\
w & \mapsto (2, 8)(6, 10, 9, 7)(11, 12)\\
e & \mapsto (1, 2)(4, 5, 11, 12)(9, 10)\\
r & \mapsto (1, 3)(6, 7, 9, 10)(11, 12) \\
0 & \mapsto (1, 4, 6)(2, 5, 7, 3, 11, 9)(8, 12, 10)\\
1 & \mapsto (1, 5, 7)(2, 4, 6, 3, 12, 10)(8, 11, 9)
\end{align*}
extend to a homomorphism $\varphi \colon \Gamma_{\mathcal B} \to \Sym(12)$. Clearly the permutations $\varphi(qr) = (1, 3, 8)(4, 11, 5)(6, 7, 9)$ and $\varphi(0)^2 = (1, 6, 4)(2, 7, 11)(3, 9, 5)(8, 10, 12)$ have order~$3$. Moreover they do not commute. By Theorem~\ref{thm:anti-torus} the elements $0$ and $qr$ span an anti-torus in $\Gamma_{\mathcal B}$, whereas Corollary~\ref{cor:birev-non-torsion} ensures that the automaton groups $G_{\mathcal B}$ and $G_{\mathcal B^*}$ both contain non-abelian free sub-semigroups (compare \cite[Proposition~3.1]{KPS}).
\end{example}
	
\begin{example}\label{ex:KPS2}
The automaton $\mathcal C$ with set of states $Q = \{a, b, c, d\}$ and alphabet $L = \{0, 1\}$  is defined on Figure~7 of \cite{KPS}. We do not reproduce that figure, but it can be reconstructed from the presentation of the associated BMW group $\Gamma_{\mathcal C}$, which is the following:
\begin{align*}
\Gamma_{\mathcal C} = \langle a, b, c, d, 0, 1| & 
%a0=1d, \ a1=0d, \ b0=0c, \ b1=1c, \\
%&  c0=0a, \ c1=1b,\  d0=0b, \ d1=1a\rangle.
a0 = 0c, \ a1=1d, \ b0 = 0d, \ b1=1c, \\
& c0=0b, \ c1 = 1b, \ d0=1a, \ d1=0a \rangle. 
\end{align*}
We claim that $\Gamma_{\mathcal C}$ admits the group $\Gamma_0$ from Corollary~\ref{cor:first-fractal-example} as a quotient. Indeed, renaming the generators of $\Gamma_0$ by $a_{0}, b_{0}, x, y, z, t$ to avoid confusion with the elements of $Q$, one checks that the assignments 
\begin{align*}
(a, b, c, d) & \mapsto (z, y, t, x)\\
(0, 1)& \mapsto (b_{ 0}, a_{0})
\end{align*}
extend to a surjective homomorphism. In fact this homomorphism factors through 	$\Gamma_{\mathcal C} /\langle \! \langle a^2, b^2, c^2, d^2\rangle \!\rangle $, and yields an isomorphism $\Gamma_{\mathcal C} /\langle \! \langle a^2, b^2, c^2, d^2\rangle \!\rangle  \simeq \Gamma_0$. The proof of Corollary~\ref{cor:first-fractal-example} shows that the hypotheses of Corollary~\ref{cor:birev-non-torsion} are fulfilled. This shows that the automaton groups $G_{\mathcal C}$ and $G_{\mathcal C^*}$ both contain non-abelian free sub-semigroups (compare \cite[Proposition~3.6]{KPS}).
\end{example}

\begin{remark}
We have mentioned in the introduction the open question from \cite[\S4.4]{Cap_survey} asking whether every irreducible lattice in a product of two trees contains an anti-torus. We have also mentioned the conjecture from  \cite{GodKli} that a group defined by a bi-reversible automaton contains an element of infinite order if and only if it is infinite. We note that a positive answer to the former question implies that the conjecture holds: indeed, in view of the work of Glasner--Mozes \cite[Proposition~2.12]{GM}, this follows from Proposition~\ref{prop:anti-torus}. 
\end{remark}

Let us now consider the tuple $\mathcal A = (V_1, V_2, a, b, c, d)$ defined in the introduction. We may naturally view it as a Mealy automaton, with $V_1$ as set of states and $V_2$ as alphabet. The transition  map is defined as $\lambda  \colon (v_1, v_2) \mapsto av_1 + b v_2$ and the output map  is $\rho  \colon (v_1, v_2) \mapsto cv_1 + d v_2$. We shall say that $\mathcal A $ is a \textbf{Mealy automaton of matrix type}. The following observation follows directly from the definitions. 

\begin{lem}\label{lem:birev}
The automaton $\mathcal A = (V_1, V_2, a, b, c, d)$ has the following properties. 
\begin{enumerate}[label=(\roman*)]
\item $\mathcal A$ is invertible if and only if the linear operator $d$ is invertible. 

\item  $\mathcal A$ is reversible if and only if the linear operator $a$ is invertible. 

\item $\mathcal A$ is bi-reversible if and only if  $a$, $d$ and $\left(\begin{matrix} a & b \\ c & d \end{matrix}\right)$ are invertible. 
\end{enumerate}
\end{lem}

\begin{remark}\label{rem:bi-reversible}
Suppose that $\mathcal A = (V_1, V_2, a, b, c, d)$ is Mealy automaton of matrix type. It is easy to check that if $\mathcal A$ is invertible, then the inverse automaton $\mathcal A^{-1}$ is also of matrix type: in the inverse automaton, the tuple $(a, b,  c , d)$  is replaced by $( a- bd^{-1} c, b d^{-1},  -d^{-1}c,  d^{-1} )$. Similarly, if  $\mathcal A$ is reversible (resp. bi-reversible), then $(\mathcal A^*)^{-1}$ (resp. $((\mathcal A^{-1})^*)^{-1}$) is of matrix type. 
\end{remark}

 In view of \cite{GM}, we infer from  Lemma~\ref{lem:birev} that   $\Gamma_{\mathcal A}$ is a BMW group as soon as  $a$, $d$ and $\left(\begin{matrix} a & b \\ c & d \end{matrix}\right)$ are invertible. We shall assume henceforth that this condition is satisfied.

\section{Flats spanned by periodic lines}

%We assume henceforth that  $a$, $d$ and $\left(\begin{matrix} a & b \\ c & d \end{matrix}\right)$ are invertible. 
The product of trees $T_1 \times T_2$ is a square complex than can be identified with the presentation complex associated with the presentation of $\Gamma_{\mathcal A}$ defined above. The $1$-skeleton of that complex is the Cayley graph of $\Gamma_{\mathcal A}$  with respect to the generating set $V_1 \sqcup V_2$. The horizontal (resp. vertical) edges are labeled by the elements of $V_1$ (resp. $V_2$), while the squares are labeled by the relators of $\Gamma_{\mathcal A}$, that are in one-to-one correspondence with the direct sum $V_1 \oplus V_2$. 

Fix $v_1 \in V_1$ and $v_2 \in V_2$. Let $\ell_1$ be the horizontal geodesic line in $T_1$ containing the base vertex $e$, that is a translation axis for the generator $v_1$  of  $\Gamma_{\mathcal A}$. Define similarly the   vertical line $\ell_2$ as a translation axis for $v_2$ passing through the base vertex $e$. If one identifies the vertex set of $T_1 \times T_2$ with  $\Gamma_{\mathcal A}$, the vertex set of $\ell_i$ is the cyclic subgroup of $\Gamma_{\mathcal A}$ generated by $v_i$. 

Let $P_{v_1, v_2} \subset T_1 \times T_2$ be the flat plane spanned in $\ell_1 \cup \ell_2$. Hence the vertex set of $P_{v_1, v_2}$ may be identified with the product $\langle v_1 \rangle \times \langle v_2 \rangle$. The square complex structure of $T_1 \times T_2$ induces a square tiling on  $P_{v_1, v_2}$. We let 
$$w \colon \mathbf Z^2 \to V_1 \oplus V_2$$ 
be the map that associates with $(m, n)$ the relator in $V_1 \oplus V_2$ that labels the square in $P_{v_1, v_2}$ spanned by the vertices with coordinates $(v_1^m, v_2^n) $, $(v_1^{m+1}, v_2^n) $, $(v_1^m, v_2^{n+1}) $ and $(v_1^{m+1}, v_2^{n+1}) $ in $\langle v_1 \rangle \times \langle v_2 \rangle$, see Figure~\ref{fig:anti-torus}. We say that $w$ is the \textbf{canonical tiling} associated with the flat $P_{v_1, v_2}$. The goal of this section is to show that $w$ is always generated by a substitution of length $p = \mathrm{char}(F)$. This will be achieved in Proposition~\ref{prop:anti-tori:substitution} below. A first step is afforded by the following. 

\begin{figure}[h]
\includegraphics[height=7cm]{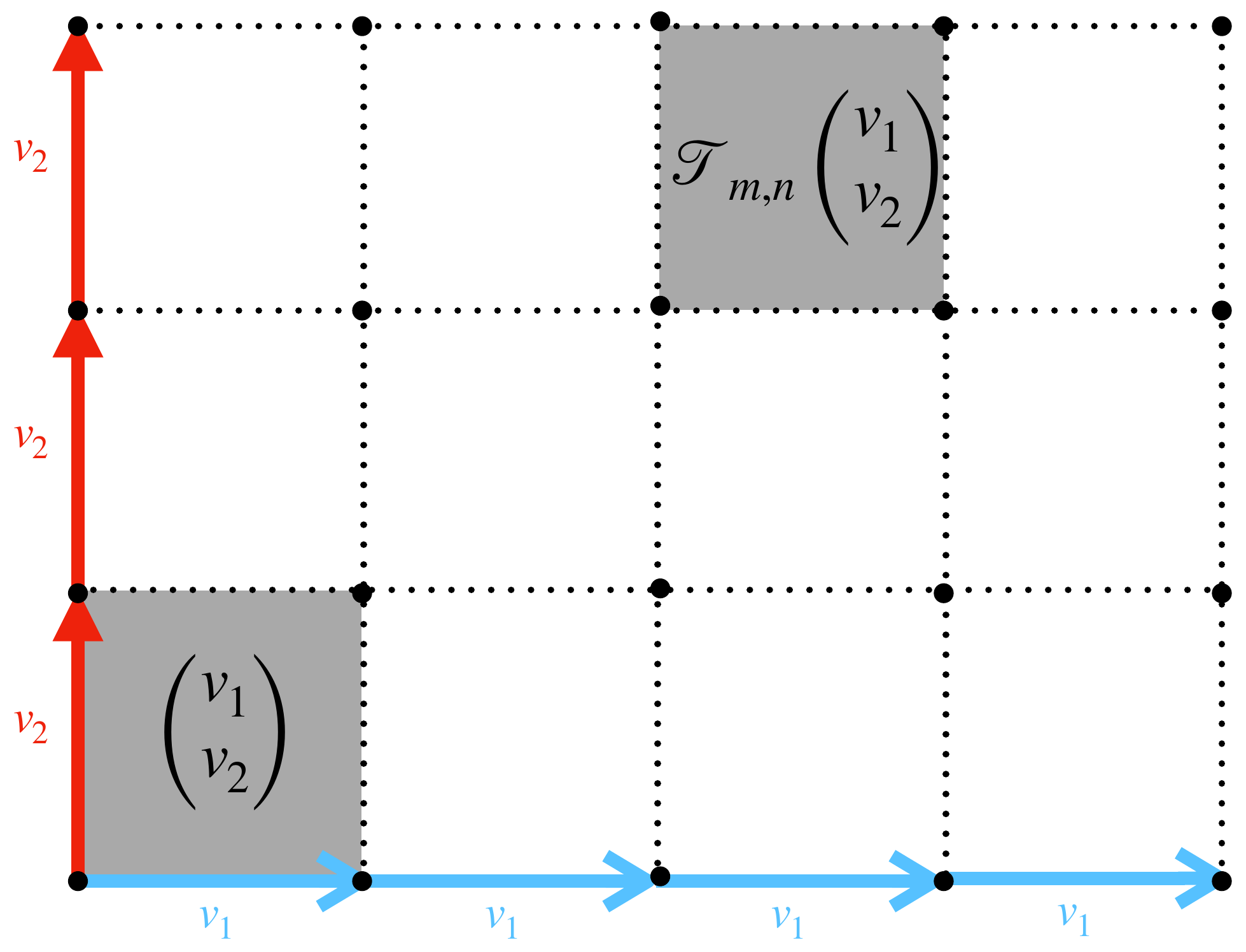}

\caption{The flat plane $P_{v_1, v_2}$, where $w(m, n) = \mathcal T_{m, n}\left(\begin{matrix} v_1 \\ v_2 \end{matrix} \right)$} \label{fig:anti-torus}
\end{figure}

\begin{prop}\label{prop:recursive}
Let $\mathcal T \colon \NN^2 \to \mathrm{End}(V_1 \oplus V_2)$ be the map recursively defined by setting $\mathcal T_{0, 0} = \mathrm{Id}$, $\mathcal T_{ 0, n} = \left(\begin{matrix} a & b \\ 0 & 1 \end{matrix}\right)^n$ for $n>0$, $\mathcal T_{m, 0} = \left(\begin{matrix} 1 & 0 \\ c & d \end{matrix}\right)^m$   for $m>0$, and 
$$
\mathcal T_{m+1, n+1} =  \left(\begin{matrix}  a & b \\ 0 & 1\end{matrix}\right)  \mathcal T_{m+1, n} 
+ \left(\begin{matrix}1 & 0 \\ c & d  \end{matrix}\right)  \mathcal T_{m, n+1} 
-  \left(\begin{matrix} a & b \\ c & d \end{matrix}\right)  \mathcal T_{m, n} 
$$ 
for all $m, n \in \mathbf N$. 

Then we have $w(m, n) =  \mathcal T_{m, n}  \left(\begin{matrix} v_1 \\ v_2\end{matrix}\right)$ for all $m, n \in \mathbf N$. 
\end{prop}
\begin{proof}
We represent the elements of $V_1 \oplus V_2$ as column vectors.% of the form $\left( \begin{matrix} w_1 \\ w_2 \end{matrix}\right)$ with $w_i \in V_i$. 

For $m=0$, it follows from a straightforward induction on $n\geq 0$ that  $w(0, n) =  \left(\begin{matrix} a & b \\ 0 & 1 \end{matrix}\right)^n \left( \begin{matrix} v_1 \\ v_2 \end{matrix}\right)$. Similarly, for $n=0$, an induction on $m$ shows that  $w(m, 0) =  \left(\begin{matrix} 1 & 0 \\ c & d \end{matrix}\right)^m \left( \begin{matrix} v_1 \\ v_2 \end{matrix}\right)$. Hence the required assertion holds for $m=0$ or $n=0$. 

To finish the proof, we shall prove by induction on $m+n$ that $w(m, n) =  \mathcal T_{m, n}  \left(\begin{matrix} v_1 \\ v_2\end{matrix}\right)$.   Fix $m, n \geq 0$. Given the previous paragraph, it suffices to show that $w(m+1, n+1) =  \mathcal T_{m+1, n+1} \left( \begin{matrix} v_1 \\ v_2 \end{matrix}\right)$, provided that the same holds for pairs of indices whose sum is less than   $m+n + 2$. 

\begin{figure}[h]
\includegraphics[height=6cm]{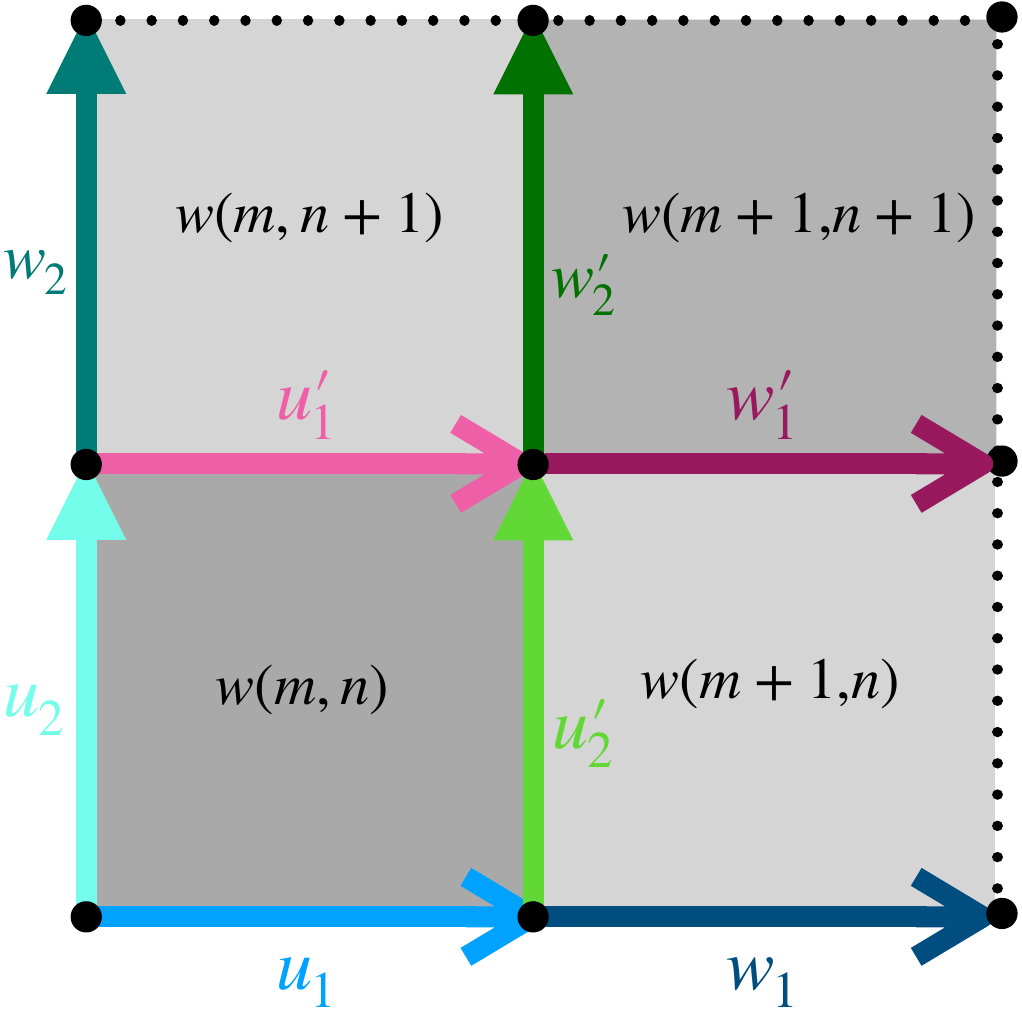}

\caption{Proof of Proposition~\ref{prop:recursive}} \label{fig:rec}
\end{figure}

As displayed on Figure~\ref{fig:rec}, we set 
\begin{align*}
w(m, n)  &= \left( \begin{matrix} u_1 \\ u_2 \end{matrix}\right),  \\
 w(m+1, n) & = \left( \begin{matrix} w_1 \\ u'_2 \end{matrix}\right),\\
 w(m, n+1) & = \left( \begin{matrix} u'_1 \\ w_2 \end{matrix}\right), \\
  w(m+1, n+1) & = \left( \begin{matrix} w'_1 \\ w'_2 \end{matrix}\right).
 \end{align*}
 In view of the induction hypothesis, we have 
  \begin{align*}
   \mathcal T_{m+1, n+1} \left( \begin{matrix} v_1 \\ v_2 \end{matrix}\right) & =
  \left[\left(\begin{matrix}  a & b \\ 0 & 1\end{matrix}\right)  \mathcal T_{m+1, n} 
+ \left(\begin{matrix}1 & 0 \\ c & d  \end{matrix}\right)  \mathcal T_{m, n+1} 
-  \left(\begin{matrix} a & b \\ c & d \end{matrix}\right)  \mathcal T_{m, n} \right]  \left( \begin{matrix} v_1 \\ v_2 \end{matrix}\right)\\
&= \left(\begin{matrix}  a & b \\ 0 & 1\end{matrix}\right) w(m+1, n) 
+ \left(\begin{matrix}1 & 0 \\ c & d  \end{matrix}\right) w(m, n+1)
-  \left(\begin{matrix} a & b \\ c & d \end{matrix}\right)w(m, n)\\
&=  \left(\begin{matrix}  a & b \\ 0 & 1\end{matrix}\right)  \left( \begin{matrix} w_1 \\ u'_2 \end{matrix}\right)
+ \left(\begin{matrix}1 & 0 \\ c & d  \end{matrix}\right)  \left( \begin{matrix} u'_1 \\ w_2 \end{matrix}\right)
-  \left(\begin{matrix} a & b \\ c & d \end{matrix}\right) \left( \begin{matrix} u_1 \\ u_2 \end{matrix}\right)\\
&= \left(\begin{matrix}  aw_1 + bu'_2 + u'_1 -a u_1 -bu_2 \\
u'_2 + cu'_1 + dw_2 -cu_1 - du_2 \end{matrix}\right).
  \end{align*}
The defining relations of $\Gamma_{\mathcal A}$ imply that 
$$u'_1 = au_1 + bu_2, \quad u'_2 =  cu_1 +d u_2, \quad w'_1 = aw_1 + bu'_2 \quad\text{and}\quad w'_2 = cu'_1 + dw_2.$$
Therefore we obtain
\begin{align*}
   \mathcal T_{m+1, n+1} \left( \begin{matrix} v_1 \\ v_2 \end{matrix}\right) 
   & =\left(\begin{matrix}  aw_1 + bu'_2  \\
cu'_1 + dw_2 \end{matrix}\right)\\
&= \left( \begin{matrix} w'_1 \\ w'_2 \end{matrix}\right)\\
&=w(m+1, n+1),
\end{align*}
as required. 
\end{proof}

We now recall a result from \cite{CapVas} showing that tilings by matrices defined via a recurrence relation as in Proposition~\ref{prop:recursive} are always generated by a substitution of length~$p$. 
		
\begin{prop}\label{prop:autom-seq}
Let $F$ be a finite field of characteristic $p$,   $k \geq 1$ be an integer and   $M = \mathrm{Mat}_{k \times k}(F)$ be the corresponding matrix algebra over $F$. Given matrices $A, B, C \in M$, we let $\mathcal T \colon \mathbf Z^2 \to M$ be the tiling recursively defined by the following conditions:
$\mathcal T(0, 0) = 1$,  $ \mathcal T(m, n) = 0$ if  $(m, n) \not \in \mathbf N^2$, and 
 \begin{align}\label{eq:rec:2}
   \mathcal T(m, n) =  
  A \mathcal T(m, n-1) + B\mathcal T(m-1, n) + C \mathcal T(m-1, n-1)
\end{align}
  if  $ (m, n)  \in \mathbf N^2 \setminus \{(0, 0)\}$.
  
Then there is a $2$-dimensional tiling $\overline{\mathcal T}$ with color set $M$, invariant under a linear substitution of length $p$, and a linear map $\tau \colon M \to M$ such that $\mathcal T = \tau \circ \overline{\mathcal T}$.  
\end{prop}
\begin{proof}
This follows from Corollary~1.6 and Remark~1.10 in \cite{CapVas}. 
\end{proof}
	
Let us now return to the canonical tiling $w \colon \ZZ^2 \to V_1 \oplus V_2$ defined by the flat plane $P_{v_1, v_2} \subset T_1 \times T_2$ as above. %associated with the pair of generators $(v_1, v_2) \in V_1 \times V_2$. 

\begin{prop}\label{prop:anti-tori:substitution}
Let $(v_1, v_2) \in V_1 \times V_2$ be a pair of generators of $\Gamma_{\mathcal A}$ and $w \colon \ZZ^2 \to V_1 \oplus V_2$ be the canonical tiling defined by the flat plane $P_{v_1, v_2} \subset T_1 \times T_2$. 

Then $w$ is generated by a substitution of length $p = \mathrm{char}(F)$. Moreover there exist $s\geq 0$ and $t \geq 1$ such that the tiling $w^{p^s}$ by blocks of size $p^s$ determined by $w$ is invariant under a substitution of length $p^t$. 
\end{prop}
\begin{proof} 
Restricting $w$ to each of the four quadrants of the plane $P_{v_1, v_2}$ determined by the axes $\ell_1$ and $\ell_2$, we obtain four tilings $\mathcal T_i \colon \ZZ^2 \to V_1 \oplus V_2$, defined as follows. The tiling $\mathcal T_1$ is defined by setting $\mathcal T_1(m, n) = w(m, n)$ if $m, n \geq 0$, and $\mathcal T_1(m, n) = 0$ else. We set $\mathcal T_2(m, n) = w(m, n)$ if $m <0 $ and $n \geq 0$, and $\mathcal T_1(m, n) = 0$ else. We set  $\mathcal T_3(m, n) = w(m, n)$ if $m <0 $ and $n < 0$, and $\mathcal T_3(m, n) = 0$ else. Finally, we set  $\mathcal T_4(m, n) = w(m, n)$ if $m \geq 0  $ and $n < 0$, and $\mathcal T_4(m, n) = 0$ else. Observe that, by definition, we have
$w = \mathcal T_1 + \mathcal T_2 + \mathcal T_3 + \mathcal T_4$. 

By Propositions~\ref{prop:recursive} and~\ref{prop:autom-seq}, the tiling $\mathcal T_1$ is generated by a substitution of length $p$. We let  $\overline{\mathcal T}_1 \colon \ZZ^2 \to M_1$ be the substitution-invariant tiling and   $\tau_1 \colon M_1 \to M$ be the linear map afforded by  those results, where $M_1$ is a copy of $M$. Moreover we have  $\mathcal T_1 = \sigma_1 \circ  \overline{\mathcal T}_1$, where $\sigma_1 \colon M_1 \to V_1 \oplus V_2$ is the map $\tau_1$ post-composed with the evaluation at the vector $\left(\begin{matrix} v_1 \\ v_2 \end{matrix}\right)$. 

A similar result applies for the tiling $\mathcal T_i$ for each $i \in \{2, 3, 4\}$. This can be established by mimicking the argument given for $i=1$. Alternatively, one may observe that $\mathcal T_2$ is the image under a reflection of a tiling of the first quadrant, associated with the inverse automaton $\mathcal A^{-1}$, which is also of matrix type (see Remark~\ref{rem:bi-reversible}). A substitution may be viewed as a map defined on all $2$-dimensional tilings with a given color set; that map commutes with the shifts and the reflections. In this way, we see that Propositions~\ref{prop:recursive} and~\ref{prop:autom-seq} also apply to $\mathcal T_2$. Similarly, they apply to $\mathcal T_3$ and $\mathcal T_4$, which are isometric images of the tilings of the first quadrant associated with the automata of matrix type $(\mathcal A^*)^{-1}$ and  $((\mathcal A^{-1})^*)^{-1}$. Therefore, for each $i=1, \dots, 4$, we obtain a tiling $\overline{\mathcal T}_i \colon \ZZ^2 \to M_i$ with color set a finite $F$-vector space $M_i$, invariant under a substitution of length $p$, and a linear map $\sigma_i \colon M_i \to V_1 \oplus V_2$, such that  $\mathcal T_i = \sigma_i \circ  \overline{\mathcal T}_i$. 

We set $W = M_1 \oplus M_2 \oplus M_3 \oplus M_4$. Viewing $M_i$ as a subspace of $W$, we may assume that $\overline{\mathcal T}_i$ is a tiling with color set $W$. We let $\mathcal L \colon \ZZ^2 \to W$ be the tiling defined by setting $\mathcal L = \sum_{i=1}^4 \overline{\mathcal T}_i$, and $\sigma \colon W \to V_1 \oplus V_2$ be the linear map that sends $(z_1, z_2, z_3, z_4)  \in W$ to $\sum_{i=1}^4 \sigma_i(z_i)$. By construction the tiling $\mathcal L$ is invariant under a substitution of length $p$. Moreover we have $w = \sigma \circ \mathcal L$. This confirms that $w$ is generated by a substitution of length $p$. 

The last assertion of the proposition concerning the tiling by blocks of size $p^s$ determined by $w$ now follows directly from \cite[Corollary~2.2]{CapVas}. 
\end{proof}

It follows that for any choice of a pair $(v_1, v_2) \in V_1 \oplus V_2$, the translation axes $\ell_1$, $\ell_2$ of the elements $v_1$, $v_2$, viewed as generators of $\Gamma_{\mathcal A}$, span a flat that is fractal, in the precise  sense defined above. Of course, it is possible that this flat is also periodic, for  example if one takes $v_1 =0$  and $v_2=0$. Thus, it remains to prove that the flat in question is actually non-periodic (and is thus an anti-torus) for a suitable choice of $(v_1, v_2) $. This is the goal of the next section. 

\section{Fractal anti-tori}

\begin{lem}\label{lem:affine-group}
Let $K$ be a field. Let $t_1, t_2 \in K^*$ be non-zero elements whose multiplicative order is infinite. Given $u_1, u_2 \in K$, the following conditions are equivalent. 
\begin{enumerate}[label=(\roman*)]
\item  $(u_1, t_1) $ and $(u_2, t_2)$ commute in the affine group $ K \rtimes K^*$. 

\item There exist integers $m_1, m_2 >0$ such that $(u_1, t_1)^{m_1} $ and $(u_2, t_2)^{m_2}$ commute in the affine group $ K \rtimes K^*$. 

\item $\frac {u_1} {1-t_1} = \frac {u_2} {1-t_2}$.

\end{enumerate}

\end{lem}
\begin{proof}
Consider the action of the affine group $ K \rtimes K^*$ on the affine line, defined by $(u, t) \colon K \to K : z \mapsto  t z  + u$. The action is sharply $2$-transitive, and the stabilizer of a point is conjugate to the multiplicative group $K^*$. Thus each non-trivial element has at most one fixed point, and two non-trivial elements fixing a point commute if and only if they fix the same point. 

If $u\neq 0$ and $t \neq 1$,  the unique fixed point of $(u, t) \in K \rtimes K^*$ is $\frac u {1-t}$. Moreover, we have $(u, t)^m = (u(1+ t + \dots + t^{m-1}), t^m) = (u \frac {1-t^m}{1-t}, t^m)$ for any integer $m \geq 0$. Thus, if the multiplicative order of $t$ is infinite, then all the positive powers of $(u, t)$ have the same fixed point. The required conclusion follows. 
\end{proof}

\begin{lem}\label{lem:polynome}
Retain the notation of Theorem~\ref{thm:main}. Let 
$$\mathbf d(x, y) = \det\left(\begin{matrix} 1-ax & bx \\ cy & 1-dy \end{matrix}\right) \in F[x, y].$$ 
The following conditions are equivalent.
\begin{enumerate}[label=(\roman*)]
\item $\mathbf d(x, y)= \det(1-ax) \det(1-dy)$.
\item There exist polynomials $f_1 \in F[x]$ and $f_2 \in F[y]$ such that $\mathbf d(x, y) = f_1(x) f_2(y)$. 
\end{enumerate}
	
\end{lem}

\begin{proof}
Suppose that (ii) holds. Then we have $ \det(1-ax)  = \mathbf d(x, 0) = f_1(x) f_2(0)$ and $ \det(1-dy) = \mathbf d(0, y) = f_1(0) f_2(y)$. Multiplying both equalities, we obtain that  $ \det(1-ax) \det(1-dy) = f_1(x) f_2(y)$ since  $ f_1(0) f_2(0) = \mathbf d(0, 0) = 1$. Hence (i) holds. The converse implication is clear. 
\end{proof}

The ring $F[x, y]$ is a unique factorization domain. In view of Lemma~\ref{lem:polynome}, the condition provided by the last hypothesis of Theorem~\ref{thm:main} ensures that the polynomial $d(x, y)$ has at least one irreducible factor $f \in F[x, y]$ that is neither in $F[x]$ nor in $F[y]$. Since $f$ is irreducible, the quotient ring $F[x, y]/ (f)$ is a domain. We define  
\begin{equation}\label{eq:K}
	K = \mathrm{Frac}(F[x, y]/ (f))
\end{equation}
as its field of fractions. We identify $x$ and $y$ with their natural images in $K$. 

\begin{lem}\label{lem:K}
$K$ is a global function field. The elements $x, y \in K$ are transcendental over $F$. In particular their multiplicative order in $K^*$ is infinite. 
\end{lem}
\begin{proof}
If $x \in K$ (resp. $y \in K$) were algebraic over $F$, it would be a root of a polynomial with coefficients in $F$ that would be divisible by $f$. This is not the case by the choice of the irreducible polynomial $f$. Thus $x$ and $y$ are transcendental over $F$. By construction $K$ is generated by $x$ and $y$. Since $y$ is algebraic over $F(x)$, this confirms that   $K$ is a finite algebraic extension of the  field $F(x)$. 
\end{proof}
	
\begin{prop}\label{prop:rep}
Let $\mathcal A = (V_1, V_2, a, b, c, d)$ satisfy the hypotheses of Theorem~\ref{thm:main} and let $K$ be the global function field defined by (\ref{eq:K}). Then there exist $F$-linear maps $\phi_1 \colon V_1 \to K$ and $\phi_2 \colon V_2 \to K$, not both equal to zero, such that the assignments
$$V_1 \to K \rtimes K^* : v_1 \mapsto (\phi_1(v_1), y) \qquad \text{and}\qquad 
V_2 \to K \rtimes K^* : v_2 \mapsto (\phi_2(v_2), x) $$
extend to a homomorphism $\rho \colon \Gamma_{\mathcal A} \to K \rtimes K^*$.
\end{prop}
\begin{proof}
In the group $\Gamma_{\mathcal A}$, each pair $(v_1, v_2) \in V_1 \oplus V_2$ satisfies the defining relation
$$v_1 \cdot (cv_1 + d v_2) = v_2 \cdot (av_1 + d v_2).$$
Moreover, those are the only defining relations of $\Gamma_{\mathcal A}$. Hence, for the assignments above to extend to a homomorphism as required,  it suffices to ensure that the image of each pair $(v_1, v_2)$ satisfies the required relation. Thus we need to construct $F$-linear maps $\phi_1 \colon V_1 \to K$ and $\phi_2 \colon V_2 \to K$ that satisfy the condition 
$$(\phi_1(v_1), y) (\phi_2(cv_1+ dv_2), x) = (\phi_2(v_2), x)(\phi_1(av_1 + bv_2), y)$$
for all $(v_1, v_2) \in V_1 \oplus V_2$. 
This is the case if and only if the equality
$$(\phi_1 + y \phi_2   c -x \phi_1 a)v_1 = (\phi_2 + x\phi_1 b - y\phi_2 d) v_2$$
holds for all $(v_1, v_2) \in V_1 \oplus V_2$, which is in turn equivalent to 
$$ \phi_1 + y \phi_2   c -x \phi_1 a = 0 \qquad \text{and} \qquad \phi_2 + x\phi_1 b - y\phi_2 d = 0.$$
Let us now identify the vector space of $F$-linear maps $V_i \to K$ with $K^{d_i}$, where $d_i = \dim(V_i)$. We shall view $(\phi_1, \phi_2)$ as a vector of   the $K$-vector space $K^{d_1} \oplus K^{d_2}$, and the matrix $\left(
\begin{matrix}
1-xa & xb\\
yc & 1-yd
\end{matrix}
\right)$ as a $K$-linear operator of that same space. Hence the last equation can be formulated in block matrix form as follows:
$$(\phi_1, \phi_2)\left(
\begin{matrix}
1-xa & xb\\
yc & 1-yd
\end{matrix}
\right) = (0, 0). $$
Hence, to ensure the existence of a non-zero pair $(\phi_1, \phi_2)$ as required, it suffices to verify that the block matrix $\left(
\begin{matrix}
1-xa & xb\\
yc & 1-yd
\end{matrix}
\right)$, viewed as  a linear operator on $K^{d_1} \oplus K^{d_2}$, has zero determinant. This is indeed the case by the very definition of $K$ (see (\ref{eq:K})). 
\end{proof}
	
\begin{cor}\label{cor:non-commuting-pair}
There exist $v_1 \in V_1$ and $v_2 \in V_2$ such that for all integers $m_1, m_2  >0$, the elements $v_1^{m_1}$ and $v_2^{m_2}$ do not commute in $\Gamma_{\mathcal A}$. 
\end{cor}
\begin{proof}
Let $\phi_1, \phi_2$ be the $F$-linear maps afforded by Proposition~\ref{prop:rep}. 

Since $\phi_1$ and $\phi_2$ are not both equal to zero, there exists $i \in \{1, 2\}$ such that $\phi_i$ is non-zero. Hence there exists $v_i \in V_i$ such that $\phi_i(v_i) \neq 0$. Let us set $t_1 = y$ and $t_2 = x$. By Lemma~\ref{lem:K}, both $t_1$ and $t_2$ have infinite multiplicative order in $K^*$. In view of Lemma~\ref{lem:affine-group}, we deduce that no positive power of $(\phi_i(v_i), t_i)$ commutes with a positive power of $(0, t_{3-i})$.  Let $v_{3-i} $ be the zero vector in $V_{3-i}$. Viewing $v_1$ and $v_2$ as generators of $\Gamma_{\mathcal A}$, we have $\rho(v_i) = (\phi_i(v_i), t_i)$ and  $\rho(v_{3-i}) = (0, t_{3-i})$ by the definition of the representation $\rho$ afforded by Proposition~\ref{prop:rep}. We infer that no positive power of $v_1$ commutes with a positive power of $v_2$. 
\end{proof}
	
\begin{proof}[Proof of Theorem~\ref{thm:main}]
Let   $v_1, v_2 \in V_1 \oplus V_2$ be the pair afforded by Corollary~\ref{cor:non-commuting-pair}. It follows from Proposition~\ref{prop:anti-torus} that they span an anti-torus. The fact that this anti-torus is fractal  follows from Proposition~\ref{prop:anti-tori:substitution}. 
\end{proof}

We finish with a clarification of the link between bi-reversible automata of matrix type and the BMW group $\Gamma_0$ from Corollary~\ref{cor:first-fractal-example}.

\begin{remark}\label{rem:Gamma0}
Set $F = \FF_2$. Let $\mathcal A_0 = (V_1, V_2, a, b, c, d)$ be the Mealy automaton of matrix type defined by setting $V_1 = V_2 = F^2$, $a = d = \left(\begin{matrix} 1& 1 \\ 0 & 1 \end{matrix}\right)$, $b = \left(\begin{matrix} 1& 0 \\ 0 & 1 \end{matrix}\right)$ and $c= \left(\begin{matrix} 1& 1 \\  1& 0 \end{matrix}\right)$. Using Lemma~\ref{lem:birev}, it is easily seen that $\mathcal A_0$ is bi-reversible; hence $\Gamma_{\mathcal A_0}$ is a BMW group. 

Since $V_1$ and $V_2$ are two copies of the same $F$-vector space, and since $\Gamma_{\mathcal A_0}$ is generated by the disjoint union $V_1 \sqcup V_2$, it is convenient to introduce the following notation that distinguishes the vectors in $V_1$ from those in $V_2$:  For $i=1, 2$, we shall view the elements of $V_i$  as column vectors represented by  symbols  $(  x_1 , x_2 )^\top_i$, with $x_1, x_2 \in F$. 

Using the respective presentations of  $\Gamma_{\mathcal A_0}$ and $\Gamma_0$, one can check that the assignments
\begin{align*}
( 0,0 )^\top_1 & \mapsto a^2 \\
( 1,0 )^\top_1 &\mapsto ab \\
(0,1 )^\top_1 &\mapsto ba \\
(1,1 )^\top_1 &\mapsto b^2 \\
( 0,0 )^\top_2 & \mapsto yt \\
( 1,0 )^\top_2 &\mapsto yx \\
(0,1 )^\top_2 &\mapsto zx\\
(1,1 )^\top_2 &\mapsto zt.
\end{align*}
extend to a homomorphism 	$\rho \colon \Gamma_{\mathcal A_0} \to \Gamma_0 $. This transports the fractal anti-tori in $\Gamma_{\mathcal A_0}$ afforded by Theorem~\ref{thm:anti-torus} to anti-tori in $\Gamma_0$. This gives a conceptual explanation for the  fractal feature of Figure~\ref{fig1} (hence also the self-similiarity of \cite[Figure~9]{GriSav}, as mentioned in the introduction).
\end{remark}
		
%For $i=1, 2$, let also $\sigma_i \colon V_i \to \Sigma_i$ be a bijection between $V_i$ and a set $\Sigma_i$. 

\bibliographystyle{amsalpha} 
\bibliography{Antitori.bib}

@book {BH,
    AUTHOR = {Bridson, Martin R. and Haefliger, Andr\'e},
     TITLE = {Metric spaces of non-positive curvature},
    SERIES = {Grundlehren der Mathematischen Wissenschaften [Fundamental
              Principles of Mathematical Sciences]},
    VOLUME = {319},
 PUBLISHER = {Springer-Verlag, Berlin},
      YEAR = {1999},
     PAGES = {xxii+643},
      ISBN = {3-540-64324-9},
   MRCLASS = {53C23 (20F65 53C70 57M07)},
  MRNUMBER = {1744486},
MRREVIEWER = {Athanase Papadopoulos},
       DOI = {10.1007/978-3-662-12494-9},
       URL = {http://dx.doi.org/10.1007/978-3-662-12494-9},
}

@book {Wise_phd,
    AUTHOR = {Wise, Daniel T.},
     TITLE = {Non-positively curved squared complexes: {A}periodic tilings
              and non-residually finite groups},
      NOTE = {Thesis (Ph.D.)--Princeton University},
 PUBLISHER = {ProQuest LLC, Ann Arbor, MI},
      YEAR = {1996},
     PAGES = {126},
      ISBN = {978-0591-07506-9},
   MRCLASS = {Thesis},
  MRNUMBER = {2694733},
       URL =
              {http://gateway.proquest.com/openurl?url_ver=Z39.88-2004&rft_val_fmt=info:ofi/fmt:kev:mtx:dissertation&res_dat=xri:pqdiss&rft_dat=xri:pqdiss:9701250},
}

@article {JW,
    AUTHOR = {Janzen, David and Wise, Daniel T.},
     TITLE = {A smallest irreducible lattice in the product of trees},
   JOURNAL = {Algebr. Geom. Topol.},
  FJOURNAL = {Algebraic \& Geometric Topology},
    VOLUME = {9},
      YEAR = {2009},
    NUMBER = {4},
     PAGES = {2191--2201},
      ISSN = {1472-2747},
   MRCLASS = {20F67 (20E08)},
  MRNUMBER = {2558308},
MRREVIEWER = {Matt T. Clay},
       DOI = {10.2140/agt.2009.9.2191},
       URL = {http://dx.doi.org/10.2140/agt.2009.9.2191},
}

@article {Wise_CSC,
	AUTHOR = {Wise, Daniel T.},
	TITLE = {Complete square complexes},
	JOURNAL = {Comment. Math. Helv.},
	FJOURNAL = {Commentarii Mathematici Helvetici. A Journal of the Swiss
	Mathematical Society},
	VOLUME = {82},
	YEAR = {2007},
	NUMBER = {4},
	PAGES = {683--724},
	ISSN = {0010-2571},
	MRCLASS = {20F67 (20F06)},
	MRNUMBER = {2341837},
	MRREVIEWER = {Richard Weidmann},
	DOI = {10.4171/CMH/107},
	URL = {http://dx.doi.org/10.4171/CMH/107},
}

@phdthesis {Radu_phd,
    AUTHOR = {Radu, Nicolas},
     TITLE = {Lattices and simple groups in trees and buildings: constructions and classifications}, 
     school={{U}niversit\'e catholique de {L}ouvain},
   year={2018},
}

@article {Rattaggi_GeomDed,
    AUTHOR = {Rattaggi, Diego},
     TITLE = {Anti-tori in square complex groups},
   JOURNAL = {Geom. Dedicata},
  FJOURNAL = {Geometriae Dedicata},
    VOLUME = {114},
      YEAR = {2005},
     PAGES = {189--207},
      ISSN = {0046-5755},
   MRCLASS = {20E05 (11R52 20F67)},
  MRNUMBER = {2174099},
MRREVIEWER = {Olivier Guichard},
       DOI = {10.1007/s10711-005-5538-9},
       URL = {http://dx.doi.org/10.1007/s10711-005-5538-9},
}

@incollection {Cap_survey,
    AUTHOR = {Caprace, Pierre-Emmanuel},
     TITLE = {Finite and infinite quotients of discrete and indiscrete
              groups},
 BOOKTITLE = {Groups {S}t {A}ndrews 2017 in {B}irmingham},
    SERIES = {London Math. Soc. Lecture Note Ser.},
    VOLUME = {455},
     PAGES = {16--69},
 PUBLISHER = {Cambridge Univ. Press, Cambridge},
      YEAR = {2019},
      ISBN = {978-1-108-72874-4},
   MRCLASS = {20F05 (20E08 20E26 20F65 20F67)},
  MRNUMBER = {3931408},
MRREVIEWER = {Vassilis\ Metaftsis},
}

@article {BoBo24,
    AUTHOR = {Bondarenko, Ievgen and Bondarenko, Nataliia},
     TITLE = {Anti-tori in quaternionic lattices over {$\Bbb F_q(t)$}},
   JOURNAL = {Algebra Discrete Math.},
  FJOURNAL = {Algebra and Discrete Mathematics},
    VOLUME = {37},
      YEAR = {2024},
    NUMBER = {2},
     PAGES = {171--180},
      ISSN = {1726-3255,2415-721X},
   MRCLASS = {20F65 (11R52 20F67)},
  MRNUMBER = {4766964},
}

@article {LPS90,
    AUTHOR = {Liebeck, Martin W. and Praeger, Cheryl E. and Saxl, Jan},
     TITLE = {The maximal factorizations of the finite simple groups and
              their automorphism groups},
   JOURNAL = {Mem. Amer. Math. Soc.},
  FJOURNAL = {Memoirs of the American Mathematical Society},
    VOLUME = {86},
      YEAR = {1990},
    NUMBER = {432},
     PAGES = {iv+151},
      ISSN = {0065-9266,1947-6221},
   MRCLASS = {20D40 (20D06 20D08 20G40)},
  MRNUMBER = {1016353},
MRREVIEWER = {Ulrich\ Dempwolff},
       DOI = {10.1090/memo/0432},
       URL = {https://doi.org/10.1090/memo/0432},
}

@article {BDR16,
    AUTHOR = {Bondarenko, Ievgen and D'Angeli, Daniele and Rodaro, Emanuele},
     TITLE = {The lamplighter group {$\Bbb{Z}_3\wr\Bbb{Z}$} generated by a
              bireversible automaton},
   JOURNAL = {Comm. Algebra},
  FJOURNAL = {Communications in Algebra},
    VOLUME = {44},
      YEAR = {2016},
    NUMBER = {12},
     PAGES = {5257--5268},
      ISSN = {0092-7872,1532-4125},
   MRCLASS = {20E22 (20E08 20F10 68Q70)},
  MRNUMBER = {3520274},
MRREVIEWER = {Alfredo\ Donno},
       DOI = {10.1080/00927872.2016.1172602},
       URL = {https://doi.org/10.1080/00927872.2016.1172602},
}

@article {GM,
    AUTHOR = {Glasner, Yair and Mozes, Shahar},
     TITLE = {Automata and square complexes},
   JOURNAL = {Geom. Dedicata},
  FJOURNAL = {Geometriae Dedicata},
    VOLUME = {111},
      YEAR = {2005},
     PAGES = {43--64},
      ISSN = {0046-5755,1572-9168},
   MRCLASS = {20M35 (20E08 20F05 37B10 37B15)},
  MRNUMBER = {2155175},
MRREVIEWER = {Alain\ Valette},
       DOI = {10.1007/s10711-004-1815-2},
       URL = {https://doi.org/10.1007/s10711-004-1815-2},
}

@unpublished{CapVas,
      title={On the self-similarity of rational power series with matrix coefficients}, 
      author={Pierre-Emmanuel Caprace and Justin Vast},
      year={2026},
      note={Preprint, arXiv:2605.22624},
      url={https://arxiv.org/abs/2605.22624}, 
}

@article {SkiSte,
    AUTHOR = {Skipper, Rachel and Steinberg, Benjamin},
     TITLE = {Lamplighter groups, bireversible automata, and rational series
              over finite rings},
   JOURNAL = {Groups Geom. Dyn.},
  FJOURNAL = {Groups, Geometry, and Dynamics},
    VOLUME = {14},
      YEAR = {2020},
    NUMBER = {2},
     PAGES = {567--589},
      ISSN = {1661-7207,1661-7215},
   MRCLASS = {20F65 (20E08 20F10)},
  MRNUMBER = {4118629},
MRREVIEWER = {Vitaly\ A.\ Roman\cprime kov},
       DOI = {10.4171/GGD/555},
       URL = {https://doi.org/10.4171/GGD/555},
}

@article {Fran,
    AUTHOR = {Francoeur, Dominik},
     TITLE = {Bireversible automata generating lamplighter groups},
   JOURNAL = {Bull. Lond. Math. Soc.},
  FJOURNAL = {Bulletin of the London Mathematical Society},
    VOLUME = {55},
      YEAR = {2023},
    NUMBER = {2},
     PAGES = {990--997},
      ISSN = {0024-6093,1469-2120},
   MRCLASS = {20F10 (20E08 20F65)},
  MRNUMBER = {4581338},
MRREVIEWER = {Enric\ Ventura Capell},
       DOI = {10.1112/blms.12772},
       URL = {https://doi.org/10.1112/blms.12772},
}

@article {ShSt,
AUTHOR = { Shallit, J. and Stolfi, J.},
     TITLE = {Two methods for generating fractals},
   JOURNAL = {Comp. Graphics},
    VOLUME = {13},
      YEAR = {1989},
     PAGES = {185--191},
}

@book {AlSh,
    AUTHOR = {Allouche, Jean-Paul and Shallit, Jeffrey},
     TITLE = {Automatic sequences},
      NOTE = {Theory, applications, generalizations},
 PUBLISHER = {Cambridge University Press, Cambridge},
      YEAR = {2003},
     PAGES = {xvi+571},
      ISBN = {0-521-82332-3},
   MRCLASS = {11B85 (11Z05 37A45 37B10 68Q45 68R15 94A45)},
  MRNUMBER = {1997038},
MRREVIEWER = {Val\'erie\ Berth\'e},
       DOI = {10.1017/CBO9780511546563},
       URL = {https://doi.org/10.1017/CBO9780511546563},
}

@book {Nekra,
    AUTHOR = {Nekrashevych, Volodymyr},
     TITLE = {Self-similar groups},
    SERIES = {Mathematical Surveys and Monographs},
    VOLUME = {117},
 PUBLISHER = {American Mathematical Society, Providence, RI},
      YEAR = {2005},
     PAGES = {xii+231},
      ISBN = {0-8218-3831-8},
   MRCLASS = {20E08 (20F65 37B15 37F10)},
  MRNUMBER = {2162164},
MRREVIEWER = {Laurent\ Bartholdi},
       DOI = {10.1090/surv/117},
       URL = {https://doi.org/10.1090/surv/117},
}

@article {Radu,
    AUTHOR = {Radu, Nicolas},
     TITLE = {New simple lattices in products of trees and their
              projections},
      NOTE = {With an appendix by Pierre-Emmanuel Caprace},
   JOURNAL = {Canad. J. Math.},
  FJOURNAL = {Canadian Journal of Mathematics. Journal Canadien de
              Math\'ematiques},
    VOLUME = {72},
      YEAR = {2020},
    NUMBER = {6},
     PAGES = {1624--1690},
      ISSN = {0008-414X,1496-4279},
   MRCLASS = {20E08 (20E32 20F65 22E40)},
  MRNUMBER = {4176704},
MRREVIEWER = {Stefan\ Kohl},
       DOI = {10.4153/s0008414x19000506},
       URL = {https://doi.org/10.4153/s0008414x19000506},
}

@article {GLNS,
    AUTHOR = {Grigorchuk, R. and Leonov, Y. and Nekrashevych, V. and
              Sushchansky, V.},
     TITLE = {Self-similar groups, automatic sequences, and unitriangular
              representations},
   JOURNAL = {Bull. Math. Sci.},
  FJOURNAL = {Bulletin of Mathematical Sciences},
    VOLUME = {6},
      YEAR = {2016},
    NUMBER = {2},
     PAGES = {231--285},
      ISSN = {1664-3607,1664-3615},
   MRCLASS = {20E08 (20F10 68Q45 68R15)},
  MRNUMBER = {3510692},
       DOI = {10.1007/s13373-015-0077-7},
       URL = {https://doi.org/10.1007/s13373-015-0077-7},
}

@incollection {BGN,
    AUTHOR = {Bartholdi, Laurent and Grigorchuk, Rostislav and Nekrashevych,
              Volodymyr},
     TITLE = {From fractal groups to fractal sets},
 BOOKTITLE = {Fractals in {G}raz 2001},
    SERIES = {Trends Math.},
     PAGES = {25--118},
 PUBLISHER = {Birkh\"auser, Basel},
      YEAR = {2003},
      ISBN = {3-7643-7006-8},
   MRCLASS = {20E08 (05B45 20F65 28A80 37F50)},
  MRNUMBER = {2091700},
MRREVIEWER = {Richard\ Kenyon},
}

@article {GriSav,
    AUTHOR = {Grigorchuk, Rostislav and Savchuk, Dmytro},
     TITLE = {Ergodic decomposition of group actions on rooted trees},
   JOURNAL = {Tr. Mat. Inst. Steklova},
  FJOURNAL = {Trudy Matematicheskogo Instituta Imeni V. A. Steklova},
    VOLUME = {292},
      YEAR = {2016},
     PAGES = {100--117},
      ISSN = {0371-9685,3034-1809},
      ISBN = {5-7846-0137-7; 978-5-7846-0137-7},
   MRCLASS = {37A15 (20E08 28D15 37A20 60A10)},
  MRNUMBER = {3628455},
MRREVIEWER = {Jan\ Kwiatkowski},
       DOI = {10.1134/S0371968516010064},
       URL = {https://doi.org/10.1134/S0371968516010064},
}

@article {FM,
    AUTHOR = {Francoeur, Dominik and Mitrofanov, Ivan},
     TITLE = {On the existence of free subsemigroups in reversible automata
              semigroups},
   JOURNAL = {Groups Geom. Dyn.},
  FJOURNAL = {Groups, Geometry, and Dynamics},
    VOLUME = {15},
      YEAR = {2021},
    NUMBER = {3},
     PAGES = {1103--1132},
      ISSN = {1661-7207,1661-7215},
   MRCLASS = {20M35},
  MRNUMBER = {4322023},
MRREVIEWER = {Alessandra\ Cherubini},
       DOI = {10.4171/ggd/626},
       URL = {https://doi.org/10.4171/ggd/626},
}

@article {GodKli,
    AUTHOR = {Godin, Thibault and Klimann, Ines},
     TITLE = {On bireversible {M}ealy automata and the {B}urnside problem},
   JOURNAL = {Theoret. Comput. Sci.},
  FJOURNAL = {Theoretical Computer Science},
    VOLUME = {707},
      YEAR = {2018},
     PAGES = {24--35},
      ISSN = {0304-3975,1879-2294},
   MRCLASS = {68Q45},
  MRNUMBER = {3734397},
MRREVIEWER = {Yu\ Qi\ Guo},
       DOI = {10.1016/j.tcs.2017.10.005},
       URL = {https://doi.org/10.1016/j.tcs.2017.10.005},
}

@article {KPS,
    AUTHOR = {Klimann, Ines and Picantin, Matthieu and Savchuk, Dmytro},
     TITLE = {Orbit automata as a new tool to attack the order problem in
              automaton groups},
   JOURNAL = {J. Algebra},
  FJOURNAL = {Journal of Algebra},
    VOLUME = {445},
      YEAR = {2016},
     PAGES = {433--457},
      ISSN = {0021-8693,1090-266X},
   MRCLASS = {20E08 (20F10 68Q70)},
  MRNUMBER = {3418065},
MRREVIEWER = {Vassilis\ Metaftsis},
       DOI = {10.1016/j.jalgebra.2015.07.003},
       URL = {https://doi.org/10.1016/j.jalgebra.2015.07.003},
}

@unpublished{CV_chains,
      title={Ascending chains of irreducible lattices, bi-reversible automata and affine arithmetic groups}, 
      author={Pierre-Emmanuel Caprace and Justin Vast},
      year={2026},
      note={Preprint, arXiv:2607.14856},
      url={https://arxiv.org/abs/2607.14856}, 
}

\end{document}